\documentclass[11pt]{amsart}
\usepackage{amssymb}
\usepackage{amsmath}
\usepackage{fancyhdr}
\usepackage[british]{babel}
\usepackage{geometry}
\usepackage{enumitem}
\usepackage{algpseudocode}
\usepackage{dsfont}
\usepackage{centernot}
\usepackage{xstring}
\usepackage{colortbl}
\usepackage[section]{algorithm}
\usepackage{graphicx}
\usepackage[font={footnotesize}]{caption}
\usepackage[usenames,dvipsnames,table]{xcolor}
\usepackage{tikz}
\usepackage[h]{esvect}
\usepackage[
		bookmarksopen=true,
		bookmarksopenlevel=1,
		colorlinks=true,
		linkcolor=darkblue,
        linktoc=page,
		citecolor=darkblue,
]{hyperref}
\usepackage{bbm}

\usepackage[capitalise]{cleveref}

\title[]{A rainbow blow-up lemma for almost optimally bounded edge-colourings}
\date{\today}

\author[S.~Ehard]{Stefan Ehard}
\address[S.~Ehard]{Institut f\"ur Optimierung und Operations Research, Universit\"at Ulm, Ulm,
Germany
}
\email{stefan.ehard@uni-ulm.de}

\author[S.~Glock]{Stefan Glock}
\author[F.~Joos]{Felix Joos}
\address[S.~Glock, F.~Joos]{School of Mathematics, University of Birmingham,
Edgbaston, Birmingham, B15 2TT, United Kingdom}
\email{[s.glock, f.joos]@bham.ac.uk}

\thanks{The research leading to these results was partially supported by the EPSRC, grant nos.~EP/N019504/1 (S.~Glock) and 
by the Deutsche Forschungsgemeinschaft (DFG, German Research Foundation) -- 339933727 (F.~Joos).
}

\newtheorem{theorem}[algorithm]{Theorem}

\newtheorem{lemma}[algorithm]{Lemma}

\newtheorem{fact}[algorithm]{Fact}

\theoremstyle{definition}

\newtheorem{conj}[algorithm]{Conjecture}
\newtheorem{defin}[algorithm]{Definition}

\newcounter{stepenv}
\newenvironment{stepenv}[1][]{\refstepcounter{stepenv}}{}

\newcounter{step}[stepenv]
\newenvironment{step}[1][]{\refstepcounter{step}\par\medskip\noindent%
        \textit{Step~\thestep. #1} \itshape\rmfamily}{\medskip}
				
\newcounter{substep}[step]
\renewcommand{\thesubstep}{\thestep.\arabic{substep}}
\newenvironment{substep}[1][]{\refstepcounter{substep}\par\medskip\noindent%
        \emph{Step~\thesubstep. #1} \itshape\rmfamily}{\medskip}

\newcounter{claim}[stepenv]
\newenvironment{claim}[1][]{\refstepcounter{claim}\par\medskip\noindent%
        \textit{Claim~\theclaim. #1} \itshape\rmfamily}{\medskip}

\numberwithin{equation}{section}

\definecolor{darkblue}{rgb}{0,0,0.5}

\def\noproof{{\unskip\nobreak\hfill\penalty50\hskip2em\hbox{}\nobreak\hfill%
       $\square$\parfillskip=0pt\finalhyphendemerits=0\par}\goodbreak}
\def\endproof{\noproof\bigskip}

\def\noclaimproof{{\unskip\nobreak\hfill\penalty50\hskip2em\hbox{}\nobreak\hfill%
       $-$\parfillskip=0pt\finalhyphendemerits=0\par}\goodbreak}

\def\endclaimproof{\noclaimproof\medskip}

\newdimen\margin
\def\textno#1&#2\par{
   \margin=\hsize
   \advance\margin by -4\parindent
          \setbox1=\hbox{\sl#1}
   \ifdim\wd1 < \margin
      $$\box1\eqno#2$$
   \else
      \bigbreak
      \hbox to \hsize{\indent$\vcenter{\advance\hsize by -3\parindent
      \it\noindent#1}\hfil#2$}
      \bigbreak
   \fi}

\def\proof{\removelastskip\penalty55\medskip\noindent\begin{stepenv}\end{stepenv}{\bf Proof. }} % in each main proof, claim and step counter set back

\def\lateproof#1{\removelastskip\penalty55\medskip\noindent\begin{stepenv}\end{stepenv}{\bf Proof of #1. }} % in each main proof, claim and step counter set back

\DeclareMathOperator{\dg}{deg}

\def\claimproof{\removelastskip\penalty55\medskip\noindent{\em Proof of claim: }}

\begin{document}

\newcommand{\new}[1]{\textcolor{red}{#1}}

\newcommand{\COMMENT}[1]{}
\newcommand{\TASK}[1]{}
\renewcommand{\TASK}[1]{\footnote{\textcolor{red!70!black}{#1}}} % comment out to hide comments

\newcommand{\todo}[1]{\begin{center}\textbf{to do:} #1 \end{center}}

\def\eps{{\varepsilon}}
\def\heps{{\hat{\varepsilon}}}

\newcommand{\ex}{\mathbb{E}}
\newcommand{\pr}{\mathbb{P}}
\newcommand{\cB}{\mathcal{B}}
\newcommand{\cA}{\mathcal{A}}
\newcommand{\cE}{\mathcal{E}}
\newcommand{\cS}{\mathcal{S}}
\newcommand{\cF}{\mathcal{F}}
\newcommand{\cG}{\mathcal{G}}
\newcommand{\bL}{\mathbb{L}}
\newcommand{\bF}{\mathbb{F}}
\newcommand{\bZ}{\mathbb{Z}}
\newcommand{\cH}{\mathcal{H}}
\newcommand{\cC}{\mathcal{C}}
\newcommand{\cM}{\mathcal{M}}
\newcommand{\bN}{\mathbb{N}}
\newcommand{\bR}{\mathbb{R}}
\def\O{\mathcal{O}}
\newcommand{\cP}{\mathcal{P}}
\newcommand{\cQ}{\mathcal{Q}}
\newcommand{\cR}{\mathcal{R}}
\newcommand{\cJ}{\mathcal{J}}
\newcommand{\cL}{\mathcal{L}}
\newcommand{\cK}{\mathcal{K}}
\newcommand{\cD}{\mathcal{D}}
\newcommand{\cI}{\mathcal{I}}
\newcommand{\cN}{\mathcal{N}}
\newcommand{\cV}{\mathcal{V}}
\newcommand{\cT}{\mathcal{T}}
\newcommand{\cU}{\mathcal{U}}
\newcommand{\cX}{\mathcal{X}}
\newcommand{\cZ}{\mathcal{Z}}
\newcommand{\cW}{\mathcal{W}}
\newcommand{\1}{{\bf 1}_{n\not\equiv \delta}}
\newcommand{\eul}{{\rm e}}
\newcommand{\Erd}{Erd\H{o}s}
\newcommand{\cupdot}{\mathbin{\mathaccent\cdot\cup}}
\newcommand{\whp}{whp }
\newcommand{\bX}{\mathcal{X}}
\newcommand{\bV}{\mathcal{V}}
\newcommand{\hbX}{\widehat{\mathcal{X}}}
\newcommand{\hbV}{\widehat{\mathcal{V}}}
\newcommand{\hA}{\widehat{A}}
\newcommand{\tA}{\widetilde{A}}
\newcommand{\hX}{\widehat{X}}
\newcommand{\hV}{\widehat{V}}
\newcommand{\tX}{\widetilde{X}}
\newcommand{\tV}{\widetilde{V}}
\newcommand{\Vc}{\overline{V}}
\newcommand{\Xc}{\overline{X}}
\newcommand{\cbI}{\overline{\mathcal{I}^\alpha}}
\newcommand{\hAj}{\widehat{A}^\sigma_j}
\newcommand{\hVM}{V^M}
\newcommand{\Haux}{H^{+}}
\newcommand{\Abad}{A^{bad}}
\newcommand{\Agood}{A'}
\newcommand{\Aupd}{A^{\sigma}}
\newcommand{\Agoodupd}{A^\ast}
\newcommand{\Gbad}{G^{bad}}
\newcommand{\Ggood}{G'}

\newcommand{\bfi}{\mathbf{i}}

\newcommand{\bd}{\mathbf{d}}

\newcommand{\doublesquig}{%
  \mathrel{%
    \vcenter{\offinterlineskip
      \ialign{##\cr$\rightsquigarrow$\cr\noalign{\kern-1.5pt}$\rightsquigarrow$\cr}%
    }%
  }%
}

\newcommand{\defn}{\emph}

\newcommand\restrict[1]{\raisebox{-.5ex}{$|$}_{#1}}

\newcommand{\prob}[1]{\mathrm{\mathbb{P}}\left[#1\right]}
\newcommand{\cprob}[2]{\prob{#1 \;\middle|\; #2}}
\newcommand{\expn}[1]{\mathrm{\mathbb{E}}\left[#1\right]}
\def\gnp{G_{n,p}}
\def\G{\mathcal{G}}
\def\lflr{\left\lfloor}
\def\rflr{\right\rfloor}
\def\lcl{\left\lceil}
\def\rcl{\right\rceil}

\newcommand{\qbinom}[2]{\binom{#1}{#2}_{\!q}}
\newcommand{\binomdim}[2]{\binom{#1}{#2}_{\!\dim}}

\newcommand{\grass}{\mathrm{Gr}}

\newcommand{\brackets}[1]{\left(#1\right)}
\def\sm{\setminus}
\newcommand{\Set}[1]{\left\{#1\right\}}
\newcommand{\set}[2]{\left\{#1\,:\;#2\right\}}
\newcommand{\krq}[2]{K^{(#1)}_{#2}}
\newcommand{\ind}[1]{$\mathbf{S}(#1)$}
\newcommand{\indcov}[1]{$(\#)_{#1}$}
\def\In{\subseteq}
\newcommand{\orb}{orb}

\begin{abstract} 
\noindent
A subgraph of an edge-coloured graph is called rainbow if all its edges have different colours. 
We prove a rainbow version of the blow-up lemma of Koml\'os, S\'ark\"ozy and Szemer\'edi that applies to almost optimally bounded colourings. A corollary of this is that there exists a rainbow copy of any bounded-degree spanning subgraph $H$ in a quasirandom host graph $G$, assuming that the edge-colouring of $G$ fulfills a boundedness condition that is asymptotically best possible. 

This has many applications beyond rainbow colourings, for example to graph decompositions, orthogonal double covers and graph labellings.
\end{abstract}

\maketitle
\section{Introduction}
We study rainbow embeddings of bounded-degree spanning subgraphs into quasirandom graphs with almost optimally bounded edge-colourings. Moreover, following the recent work of Montgomery, Pokrovskiy and Sudakov~\cite{MPS:18} on embedding rainbow trees, we present several applications to graph decompositions, graph labellings and orthogonal double covers.

Given a (not necessarily proper) edge-colouring of a graph, a subgraph is called \defn{rainbow} if all its edges have different colours. 
Rainbow colourings appear in many different contexts of combinatorics, and many problems beyond graph colouring can be translated into a rainbow subgraph problem. What makes this concept so versatile is that it can be used to find `conflict-free' subgraphs. More precisely, an edge-colouring of a graph $G$ can be interpreted as a system of conflicts on $E(G)$, where two edges conflict if they have the same colour. A subgraph is then conflict-free if and only if it is rainbow. 
For instance, rainbow matchings in $K_{n,n}$ can be used to model transversals in Latin squares. The study of Latin squares dates back to the work of Euler in the 18th century and has since been a fascinating and fruitful area of research. The famous Ryser--Brualdi--Stein conjecture asserts that every $n\times n$ Latin square has a partial transversal of size~$n-1$, which is equivalent to saying that any proper $n$-edge-colouring of $K_{n,n}$ admits a rainbow matching of size~$n-1$.

As a second example, we consider a powerful application of rainbow colourings to graph decompositions. Graph decomposition problems are central problems in graph theory with a long history, and many fundamental questions are still unanswered. We say that \defn{$H_1,\dots,H_t$ decompose $G$} if $H_1,\dots,H_t$ are edge-disjoint subgraphs of $G$ covering every edge of~$G$. Perhaps one of the oldest decomposition results is Walecki's theorem from 1892 saying that $K_{2n+1}$ can be decomposed into Hamilton cycles. His construction not only gives any decomposition, but a `cyclic' decomposition based on a rotation technique, by finding one Hamilton cycle $H^\ast$ in $K_{2n+1}$ and a permutation $\pi$ on $V(K_{2n+1})$ such that the permuted copies $\pi^i(H^\ast)$ of $H^\ast$ for $i=0,\dots,n-1$ are pairwise edge-disjoint (and thus decompose $K_{2n+1}$). The difficulty here is of course finding $H^\ast$ given $\pi$, or vice versa. Unfortunately, for many other decomposition problems, this is not as easy, or indeed not possible at all. In recent years, some exciting progress has been made in the area of (hyper-)graph decompositions, for example Keevash's proof of the Existence conjecture~\cite{keevash:14} and generalizations thereof~\cite{GKLO:16,GKLO:17,keevash:18b}, progress on the Gy\'arf\'as--Lehel tree-packing conjecture~\cite{ABCT:19,JKKO:ta} and the resolution of the Oberwolfach problem~\cite{GJKKO:18}. Those results are based on very different techniques, such as absorbing-type methods, randomised constructions and variations of Szemeredi's regularity technique.
In a recent paper, Montgomery, Pokrovskiy and Sudakov~\cite{MPS:18} brought the use of the rotation technique back into focus when proving an old conjecture of Ringel approximately, by reducing it to a rainbow embedding problem. A similar approach has previously been used by Drmota and Llad\'o~\cite{DL:14} in connection with a bipartite version of Ringel's conjecture posed by Graham and H\"aggkvist.
Ringel conjectured in 1963 that $K_{2n+1}$ can be decomposed into $2n+1$ copies of any given tree with $n$ edges. A strengthening of Ringel's conjecture is due to Kotzig~\cite{kotzig:73}, who conjectured in 1973 that there even exists a cyclic decomposition. This can be phrased as a rainbow embedding problem as follows: Order the vertices of $K_{2n+1}$ cyclically and colour each edge $\Set{i,j}\in E(K_{2n+1})$ with its distance (that is, the distance of $i,j$ in the cyclic ordering), which is a number between $1$ and~$n$. The simple but crucial observation is that if $T$ is a rainbow subtree, then $T$ can be rotated according to the cyclic vertex ordering, yielding $2n+1$ edge-disjoint copies of~$T$ (and thus a cyclic decomposition if $T$ has $n$ edges). Note that for each vertex $v$ and any given distance, there are only two vertices which have exactly this distance from~$v$. More generally, an edge-colouring is called \defn{locally $k$-bounded} if each colour class has maximum degree at most~$k$. The following statement thus implies Kotzig's and Ringel's conjecture: Any locally $2$-bounded edge-colouring of $K_{2n+1}$ contains a rainbow copy of any tree with $n$~edges.
Montgomery, Pokrovskiy and Sudakov~\cite{MPS:18} proved the following asymptotic version of this statement, which in turn yields asymptotic versions of these conjectures (all asymptotic terms are considered as $n\to \infty$).

\begin{theorem}[\cite{MPS:18}]\label{thm:MPS}
For fixed $k$, any locally $k$-bounded edge-colouring of~$K_n$ contains a rainbow copy of any tree with $(1-o(1))n/k$ edges.
\end{theorem}
Our main results are very similar in spirit. 
Roughly speaking, instead of dealing with trees, our results apply to general graphs~$H$, but we require $H$ to have bounded degree, whereas one of the great achievements of~\cite{MPS:18} is that no such requirement is necessary when dealing with trees. 
The following is a special case of our main result (Theorem~\ref{thm:quasirandom}). An edge-colouring is called \defn{(globally) $k$-bounded} if any colour appears at most $k$~times. 
\begin{theorem}\label{thm:main simplified}
Suppose $H$ is a graph on at most $n$ vertices with $\Delta(H)= \O(1)$.
Then any locally $\O(1)$-bounded and globally $(1-o(1)){\binom{n}{2}}/{e(H)}$-bounded edge-colouring of~$K_n$ contains a rainbow copy of~$H$.
\end{theorem}
It is plain that any locally $k$-bounded colouring is (globally) $kn/2$-bounded. Thus, Theorem~\ref{thm:main simplified} implies Theorem~\ref{thm:MPS} for bounded-degree trees.
Note that the assumption that the colouring is $(1-o(1)){\binom{n}{2}}/{e(H)}$-bounded is asymptotically best possible in the sense that if the colouring was not ${\binom{n}{2}}/{e(H)}$-bounded, there might be less than $e(H)$ colours, making the existence of a rainbow copy of~$H$ impossible. 

Beyond the approximate solution of Ringel's conjecture, Montgomery, Pokrovskiy and Sudakov also provide applications of their result to graph labelling and orthogonal double covers. Our applications are very much inspired by theirs and essentially proved analogously. We refer the discussion of these applications to Section~\ref{sec:apps}.

Rainbow embedding problems have also been extensively studied for their own sake. For instance, Erd\H{o}s and Stein asked for the maximal $k$
such that any $k$-bounded edge-colouring of $K_n$ contains a rainbow Hamilton cycle (cf.~\cite{ENR:83}).
After several subsequent improvements, Albert, Frieze and Reed~\cite{AFR:95} showed that $k=\Omega(n)$. 
Theorem~\ref{thm:main simplified} implies that under the additional assumption that the colouring is locally $\O(1)$-bounded, we have $k=(1-o(1))n/2$, which is essentially best possible. This is not a new result but also follows from results in~\cite{KKKO:18,MPS:19}. However, the results in~\cite{KKKO:18,MPS:19} are limited to finding Hamilton cycles or $F$-factors (in fact, approximate decompositions into these structures). Theorem~\ref{thm:main simplified} allows the same conclusion if we seek an $\sqrt{n/2}\times \sqrt{n/2}$ grid, say, or any other bounded-degree graph with roughly $n$~edges. For general subgraphs~$H$, the best previous result is due to B\"ottcher, Kohayakawa and Procacci~\cite{BKP:12}, who showed that given any $n/(51\Delta^2)$-bounded edge-colouring of $K_n$ and any graph $H$ on $n$ vertices with $\Delta(H)\le \Delta$, one can find a rainbow copy of~$H$. Our Theorem~\ref{thm:main simplified} improves, for bounded-degree graphs, the global boundedness condition to an asymptotically best possible one, under the additional assumption that the colouring is locally $\O(1)$-bounded.
\COMMENT{It would be very interesting to remove this additional assumption. With some additional ideas, it might be possible to use our method to replace $\O(1)$ with $n/\log^{\O(1)}n$, but eliminating such a condition completely would likely require new methods.}

\subsection{Main result}

We now state a more general version of Theorem~\ref{thm:main simplified}. We say that a graph~$G$ on $n$ vertices is \emph{$(\eps,d)$-quasirandom}
if for all $v\in V(G)$ we have $\dg_G(v)=(d\pm\eps)n$, and for all disjoint $S,T\subseteq V(G)$ with $|S|,|T|\geq \eps n$,
we have $e_G(S,T)= (d\pm \eps) |S||T|$.
%We say that a graph $G$ on $n$ vertices is \defn{$(\eps,d)$-quasirandom} if for all $u,v\in V(G)$ we have $\dg_G(v)=(d\pm\eps)n$ and $|N_G(u)\cap N_G(v)|=(d\pm\eps)^2n.$

\begin{theorem}\label{thm:quasirandom}
For all $d,\gamma\in (0,1]$ and $\Delta,\Lambda\in \bN$, there exist $\eps>0$ and $n_0\in\mathbb{N}$ such that the following holds for all $n\geq n_0$.
Suppose $G$ and $H$ are graphs on $n$ vertices, $G$ is $(\eps,d)$-quasirandom and $\Delta(H)\leq\Delta$. 
Then given any locally $\Lambda$-bounded and globally $(1-\gamma){e(G)}/{e(H)}$-bounded edge-colouring of~$G$, there is a rainbow copy of~$H$ in~$G$.
\end{theorem}

Clearly, Theorem~\ref{thm:quasirandom} implies Theorem~\ref{thm:main simplified}. We derive Theorem~\ref{thm:quasirandom} from an even more general `blow-up lemma' (Lemma~\ref{lem:main}). The original blow-up lemma of Koml\'os, S\'ark\"ozy and Szemer\'edi~\cite{KSS:97} developed roughly 20~years ago, is a powerful tool to find spanning subgraphs and has found numerous important applications in extremal combinatorics~\cite{BST:09,GJKKO:18,KSS:98,KSS:98a,KSS:01,KO:09,KO:13}. Roughly speaking, it says that given a $k$-partite graph $G$ that is `super-regular' between any two vertex classes, and a $k$-partite bounded-degree graph $H$ with a matching vertex partition, then $H$ is a subgraph of~$G$.
Note that the conclusion is trivial if $G$ is complete $k$-partite, so the crux here is that instead of requiring $G$ to be complete between any two vertex classes, super-regularity suffices. Such a scenario can often be obtained in conjunction with Szemer\'edi's regularity lemma, which makes it widely applicable. 
Many variations of the blow-up lemma have been obtained over the years~(e.g.~\cite{ABHKP:16,BKTW:15,csaba:07,keevash:11,KKOT:19,RR:99}). Recently, the second and third author~\cite{GJ:18} proved a rainbow blow-up lemma for $o(n)$-bounded edge-colourings which allows to find a rainbow embedding of~$H$. The present paper builds upon this result. The key novelty is that instead of requiring the colouring to be $o(n)$-bounded, our new result applies for almost optimally bounded colourings. (But we assume here that the colouring is locally $\O(1)$-bounded, which is not necessary in~\cite{GJ:18}).

In order to state our new rainbow blow-up lemma, we need to introduce some terminology.
If $c\colon E(G)\to C$ is an edge-colouring of a graph $G$ and $\alpha\in C$, denote by $e^\alpha(G)$ the number of $\alpha$-coloured edges of~$G$. Moreover, for disjoint $S,T\In V(G)$, denote by $e^\alpha_G(S,T)$ the number of $\alpha$-coloured edges of $G$ with one endpoint in $S$ and the other one in $T$.
Define $d_G(S,T):={e_G(S,T)}/{|S||T|}$ as the \defn{density} of the pair $S,T$ in $G$.
We say that the bipartite graph $G$ with vertex classes $(V_1,V_2)$ is \defn{$(\eps,d)$-super-regular}
if
\begin{itemize}
\item for all $S\In V_1$ and $T\In V_2$ with $|S|\ge \eps|V_1|$, $|T|\ge \eps |V_2|$, we have $d_G(S,T)=d\pm \eps$;
\item for all $i\in [2]$ and $v\in V_i$, we have $\dg_G(v)= (d\pm\eps)|V_{3-i}|$.
\end{itemize}
We say that $(H,G, (X_i)_{i\in[r]}, (V_i)_{i\in[r]})$ is an \emph{$(\eps,d)$-super-regular blow-up instance} if 
\begin{itemize}
\item $H$ and $G$ are graphs, $(X_i)_{i\in[r]}$ is a partition of $V(H)$ into independent sets, $(V_i)_{i\in[r]}$ is a partition of $V(G)$, and $|X_i|=|V_i|$ for all $i\in[r]$, and
\item for all $ij\in\binom{[r]}{2}$, the bipartite graph $G[V_i,V_j]$ is $(\eps,d)$-super-regular.
\end{itemize}
We say that $\phi\colon V(H)\to V(G)$ is an \defn{embedding of $H$ into $G$} if $\phi$ is injective and $\phi(x)\phi(y)\in E(G)$ for all $xy\in E(H)$. We also write $\phi\colon H\to G$ in this case. We say that $\phi$ is \defn{rainbow} if $\phi(H)$ is rainbow.

We now state our new rainbow blow-up lemma.

\begin{lemma}[Rainbow blow-up lemma]\label{lem:main}
For all $d,\gamma\in (0,1]$ and $\Delta,\Lambda,r\in \bN$, there exist $\eps>0$ and $n_0\in\mathbb{N}$ such that the following holds for all $n\geq n_0$.
Suppose $(H,G, (X_i)_{i\in[r]}, (V_i)_{i\in[r]})$ is an $(\eps,d)$-super-regular blow-up instance.
Assume further that 
\begin{enumerate}[label={\rm (\roman*)}]
	\item $\Delta(H)\leq \Delta$;
	\item $|V_i|=(1\pm\eps)n$ for all $i\in[r]$;
	\item\label{cond3 main} $c\colon E(G)\to C$ is a locally $\Lambda$-bounded edge-colouring such that the following holds for all $\alpha\in C$:
	$$\sum_{ij\in \binom{[r]}{2}}e^\alpha_G(V_i,V_j)e_H(X_i,X_j)\leq (1-\gamma)dn^2.$$\label{blow up lemma boundedness}
\end{enumerate}
Then there exists a rainbow embedding $\phi$ of $H$ into $G$ such that $\phi(x)\in V_i$ for all $i\in[r]$ and $x\in X_i$.
\end{lemma}

The boundedness condition in~\ref{blow up lemma boundedness} can often be simplified, for instance in the following natural situations: 
if $e_H(X_i,X_j)$ is the same for all pairs $i,j$, then $c$ needs to be $(1-\gamma)e(G[V_1,\dots,V_r])/e(H)$-bounded.
Similarly, if $c$ is `colour-split', that is, $e^\alpha_G(V_i,V_j)\in \Set{e^\alpha(G),0}$, then $c$ needs to be $(1-\gamma)e(G[V_i,V_j])/e(H[X_i,X_j])$-bounded for all $ij\in \binom{[r]}{2}$.
Both conditions are easily seen to be asymptotically best possible.
Condition~\ref{blow up lemma boundedness} is designed to work in the general setting of Lemma~\ref{lem:main}. In the proof of Theorem~\ref{thm:quasirandom}, we will randomly partition $V(G)$ into equal-sized $(V_i)_{i\in[r]}$ and see that~\ref{blow up lemma boundedness} holds.

%The blow-up lemma for $o(n)$-bounded colourings was applied in~\cite{GJ:18} to transfer the bandwidth theorem to the rainbow setting, using Szemer\'edi's regularity lemma. Unfortunately, it seems much more complicated to use the regularity lemma in the optimally bounded setting.

\section{Applications} \label{sec:apps}

In this section, we discuss applications of our main result to graph decompositions, graph labelling and orthogonal double covers. 
As mentioned before, these applications are inspired by recent work of Montgomery, Pokrovskiy and Sudakov~\cite{MPS:18}, and basically transfer their applications from trees to general, yet bounded degree, graphs.
%The main idea in each case is to reduce the respective problem to a rainbow embedding problem, and then to apply Theorem~\ref{thm:main simplified}. 

\subsection*{Graph decompositions}
%Rainbow edge-colourings can be utilized to find decompositions of graphs into smaller graphs.
%This has been demonstrated in the recent literature by Montgomery, Pokrovskiy and Sudakov~\cite{MPS:18} for an approximate solution of Ringel's conjecture,
%and by Drmota and Llad\'o~\cite{DL:14} regarding a conjecture of Graham and H\"aggkvist.
We briefly explain the general idea of utilizing rainbow edge-colourings to find graph decompositions, and then give two examples.

Suppose $G$ is a graph and $\Gamma$ is a subgroup of the automorphism group $Aut(G)$. 
If for some subgraph $H$ of~$G$, $\Set{\phi(H)}_{\phi\in \Gamma}$ is a collection of edge-disjoint subgraphs of $G$, we call this a \defn{$\Gamma$-generated $H$-packing in~$G$}, and if every edge of $G$ is covered, then it is a \defn{$\Gamma$-generated $H$-decomposition of~$G$}. For instance, in Walecki's theorem, $G$ is the complete graph and $\Gamma$ is generated by one permutation~$\pi$. We say that a packing/decomposition of $K_n$ is \defn{cyclic} if $\Gamma$ is isomorphic to~$\bZ_n$. 
Recall Kotzig's conjecture that for any given tree $T$ with $n$~edges, there exists a cyclic $T$-decomposition of~$K_{2n+1}$. Note that there are two natural divisibility conditions for the existence of such a decomposition, one `global' edge divisibility condition and one `local' degree condition. First, the number of edges of $K_{2n+1}$ is $(2n+1)n$ which is divisible by~$n$. Secondly, every vertex of $K_{2n+1}$ is supposed to play the role of every vertex of $T$ exactly once, thus we need that $\sum_{v\in V(T)}d_T(v)=2n$, which is true by the hand-shaking lemma. However, note that we have not used the fact that $T$ is a tree. The same divisibility conditions hold for any graph with $n$~edges.
We thus propose the following conjecture as an analogue to Kotzig's conjecture for general (bounded degree) graphs.

\begin{conj}
For all $\Delta\in \bN$, there exists $n_0$ such that for all $n\ge n_0$, the following is true. For any graph $H$ with $n$ edges and $\Delta(H)\le \Delta$, there exists a cyclic $H$-decomposition of~$K_{2n+1}$.\COMMENT{Technically, could also assume that $H$ has at most $2n+1$ vertices, but that makes the statement less pretty and follows anyways if we delete some isolated vertices.}
\end{conj}

We will provide some evidence for this conjecture below (Theorem~\ref{thm:decomp K_n}).
Before, we discuss in a general way how to use rainbow embeddings to find $\Gamma$-generated packings and decompositions. Let $G$ and $\Gamma$ be as above.
Then $\Gamma$ acts on $G$ as a group action and every element $\phi\in \Gamma$ sends vertices onto vertices and edges onto edges.
The \defn{orbit} $\Gamma \cdot e$ of an edge $e$ is defined as $\Gamma \cdot e:=\{\phi(e)\colon \phi\in \Gamma\}$.
It is well-known that two orbits are either disjoint or equal.
Hence we may colour the edges of $G$ according to which orbit they belong to.
We refer to the \defn{orbit colouring} $c_o^\Gamma$ of $G$ induced by $\Gamma$
and define $c_o^\Gamma(e):=\Gamma \cdot e$ for all $e\in E(G)$.

The following simple lemma now asserts that if we can find a rainbow copy with respect to the orbit colouring, and all orbits have maximum size, then the copies of $H$ obtained via $\Gamma$ are pairwise edge-disjoint. The proof is immediate and thus omitted.

\begin{lemma}\label{lem: group action}
	Let $G$ be a graph and let $\Gamma$ be a subgroup of $Aut(G)$ such that $|\Gamma \cdot e|=|\Gamma|$ for all $e\in E(G)$.
	Suppose that $H$ is a rainbow subgraph in $G$ with respect to~$c_o^\Gamma$.
	Then $\Set{\phi(H)}_{\phi\in \Gamma}$ is a $\Gamma$-generated $H$-packing in~$G$.
\end{lemma}
\COMMENT{
\proof
Suppose $H'$ is a rainbow copy of $H$ in $G$ with respect to $c_o^\Gamma$.
We claim that $\{\phi(H')\colon \phi \in \Gamma\}$ is a collection of $|\Gamma|$ edge-disjoint copies of $H$ in $G$.
Observe first that $\phi(H')$ is a copy of $H$ for each $\phi\in \Gamma$ because $\phi\in Aut(G)$.
Suppose that $e\in E(\phi(H')),e'\in E(\phi'(H'))$ for two distinct $\phi,\phi'\in \Gamma$.
We suppose for a contradiction that $e=e'$.
By construction, there exist $f,f'\in E(H')$ with $\phi(f)=e=e'=\phi'(f')$.
Consequently, $\Gamma \cdot f\cap \Gamma \cdot f'\neq \emptyset$ and so $\Gamma \cdot f=\Gamma \cdot f'$. 
Hence, as $H'$ is a rainbow subgraph with respect to $c_o^\Gamma$,
we conclude that $f=f'$.
However, as $|\Gamma \cdot f|=|\Gamma|$, we obtain that $\phi(f)\neq \phi'(f)=\phi'(f')$, which is a contradiction and completes the proof.
\endproof}
In particular, if $|\Gamma|=e(G)/e(H)$, then this yields a $\Gamma$-generated $H$-decomposition of~$G$.

\begin{theorem}\label{thm:decomp K_n}
	For all $\Delta\in \bN$, there exist $\eps>0$ and $n_0\in\mathbb{N}$ such that the following holds for all $n\geq n_0$.
	Suppose $H$ is a graph with $|V(H)|\leq  n$, $\Delta(H)\leq\Delta$ and at most $(1-\eps)n/2$ edges.
	Then $K_n$ contains a cyclic $H$-packing.
\end{theorem}
\proof
Let $G$ be the graph on vertex set $[n]$ that is the complete graph if $n$ is odd and is otherwise obtained from the complete graph by deleting the edges $\{i,i+n/2\}$ for all $i\in [n/2]$.
Consider the subgroup $\Gamma$ of $Aut(G)$ that is generated by the automorphism which sends a vertex $i$ to $i+1$ (modulo $n$).
Clearly, $\Gamma \cong \bZ_{n}$ and hence $|\Gamma|=n$.
In addition, $|\Gamma \cdot e|=n$ for all $e\in E(G)$ and $c_o^\Gamma$ is locally $2$-bounded.
Therefore, Theorem~\ref{thm:quasirandom} yields a rainbow copy of $H$ with respect to $c_o^\Gamma$ in $G$, which by Lemma~\ref{lem: group action} yields a cyclic $H$-packing in $G\In K_n$.
\endproof

We can also deduce a partite version of this. For simplicity, we only consider the bipartite case.

\begin{theorem}\label{thm:decomp multi}
	For all $\Delta\in \bN$, there exist $\eps>0$ and $n_0\in\mathbb{N}$ such that the following holds for all $n\geq n_0$.
	Suppose $H$ is a graph with $\Delta(H)\leq\Delta$ and at most $(1-\eps)n$ edges, and $V(H)$ is partitioned into $2$ independent sets of size~$n$.
	Then the complete bipartite graph $K_{n,n}$ contains a $\bZ_n$-generated $H$-packing.
\end{theorem}

\proof
We proceed similarly as in Theorem~\ref{thm:decomp K_n}.
Let $K_{n,n}$ have vertex set $\set{(1,i),(2,i)}{i\in[n]}$ and edge set $\set{(1,i)(2,j)}{i,j\in[n]}$. 
Consider the subgroup $\Gamma$ of $Aut(G)$ that is generated by the automorphism which sends a vertex $(\ell,i)$ to $(\ell,i+1)$ (modulo $n$ in the second coordinate), for $\ell\in[2]$. 
Consequently, $\Gamma\cong \bZ_{n}$. 
Moreover, $|\Gamma \cdot e|=n$ for all $e\in E(K_{n,n})$ and $c_o^\Gamma$ is proper. Thus, Lemma~\ref{lem:main} yields a rainbow copy of $H$ in $K_{n,n}$ with respect to $c_o^\Gamma$.
Then Lemma~\ref{lem: group action} completes the proof.
\endproof

These results demonstrate the usefulness of rainbow embeddings to decomposition problems. Clearly, the application is limited to decompositions of a host graph into copies of the same graph~$H$. 
Approximate decomposition results which do not arise from a group action but from random procedures have been studied recently in great depth.
At the expense that one does not obtain very symmetric (approximate) decompositions,
it is possible to embed different graphs and not only many copies of a single graph.
In particular, the blow-up lemma for approximate decompositions by Kim, K\"uhn, Osthus and Tyomkyn~\cite{KKOT:19} yields approximate decompositions into bounded degree graphs of quasirandom multipartite graphs.
Both this and another recent result of Allen, B\"ottcher, Hladk\'y and Piguet~\cite{ABHP:17} imply Conjecture 2.1 asymptotically for non-cyclic decompositions.

\subsection*{Orthogonal double covers}

%In the previous section, we asked for a collection of edge-disjoint graphs $H_1,\ldots,H_\ell$ that cover every edge of some host graph~$G$.
An \defn{orthogonal double cover} of $K_n$ by some graph $F$ is a collection of $n$ copies of $F$ in $K_n$ such that every edge of $K_n$ is contained in exactly two copies, and each two copies have exactly one edge in common. Note that $F$ must have exactly $n-1$ edges. For instance, an orthogonal double cover of $K_{\binom{k}{2}+1}$ by~$K_{k}$ is equivalent to a \defn{biplane}, which is, roughly speaking, the orthogonal double cover version of a finite projective plane. Only a handful of such biplanes is known and it is a major open question whether there are infinitely many.

Another natural candidate for $F$ is a spanning tree. Gronau, Mullin, Rosa conjectured the following.

\begin{conj}[Gronau, Mullin, Rosa~\cite{GMR:97}]\label{conj:ODC}
	Let $T$ be an arbitrary tree with $n$ vertices, $n \geq  2$, where $T$ is not the path of length $3$. 
	Then there exists an orthogonal double cover of $K_n$ by~$T$.
\end{conj}
Montgomery, Pokrovskiy and Sudakov~\cite{MPS:18} proved an asymptotic version of this when $n$ is a power of~$2$, using their Theorem~\ref{thm:MPS}. Similarly, our main theorem yields approximate orthogonal double covers by copies of any bounded degree graph with $(1-o(1))n$ edges whenever $n$ is a power of $2$. We omit the proof as it is verbatim the same as in~\cite{MPS:18}.

\begin{theorem}\label{thm:odc}
	For all $\Delta\in \bN$, there exist $\eps>0$ and $n_0\in\mathbb{N}$ such that the following holds for all $n\geq n_0$ with $n=2^k$ for some $k\in\mathbb{N}$.
	Suppose $H$ is a graph with $|V(H)|\leq  n$, $\Delta(H)\leq\Delta$ and at most $(1-\eps)n$ edges.
	Then the complete graph $K_n$ contains $n$ copies of $H$ such that every edge of $K_n$ belongs to at most two copies, and any two copies have at most one edge in common.
\end{theorem}

\COMMENT{\proof
We identify $V(K_n)$ with the elements of the group $\mathbb{Z}_2^k$. 
We consider the following edge-colouring $c\colon E(K_n)\to\mathbb{Z}_2^k$ with $c(ij)=i+j\in\mathbb{Z}_2^k$. This is proper and thus $n/2$-bounded.
It holds that 
$$e^\alpha(K_n)\frac{e(H)}{e(K_n)}\leq 1-\eps,$$
for every colour $\alpha\in\mathbb{Z}_2^k$.
Hence, by Theorem~\ref{thm:main simplified}, $K_n$ contains a rainbow copy $H_0$ of $H$.
For any $z\in\mathbb{Z}_2^k$, 
we define a permutation $\phi_z\colon V(K_n)\to V(K_n)$ by $\phi_z(v)=v+z$ with addition in~$\mathbb{Z}_2^k$.
Let $\phi_z(H_0)$ denote the subgraph of $K_n$ defined by 
$$\phi_z(H_0):=(\{\phi_z(v)\colon v\in V(H_0) \},\{\phi_z(u)\phi_z(v)\colon uv\in E(H_0)\}).$$
Note that $\phi_z(H_0)$ is isomorphic to $H_0$ since $\phi_z$ is a permutation of $V(K_n)$.
We claim that the collection of $n$ graphs $\cH:=\{\phi_z(H_0)\colon z\in\mathbb{Z}_2^k \}$ satisfies the theorem.
Note that $\phi_z$ preserves the colouring of the edges because $\phi_z(u)+\phi_z(v)=u+z+v+z=u+v$. 
Hence, all graphs in $\cH$ are rainbow.
Note that $\phi_z$ only keeps the edges fixed that are coloured with $z$ because $\{\phi_z(v),\phi_z(v+z) \}=\{v+z,v\}$ for any $v\in V(K_n)$.
Furthermore, $\phi_{z-z'}\circ\phi_{z'}=\phi_z$.
Combining these, we obtain that if a $z$-coloured edge belongs to $\phi_x(H_0)$ and $\phi_y(H_0)$ then $x+y=z\in\mathbb{Z}_2^k$.
Hence, every edge of $K_n$ belongs to at most two copies of $\cH$, and any two copies of $\cH$ have at most one edge in common.
\endproof}

\subsection*{Graph labellings}
The study of graph labellings began in the 1960s and has since produced a vast amount of different concepts, results and applications (see e.g. the survey~\cite{gallian:98}).
Perhaps the most popular types of labellings are graceful labellings and harmonious labellings. The former were introduced by Rosa~\cite{rosa:67} in~1967. 
Given a graph $H$ with $q$ edges, a \defn{graceful labelling of~$H$} is an injection $f\colon V(H)\to [q+1]$ such that the induced edge labels $|f(x)-f(y)|$, $xy\in E(H)$, are pairwise distinct, and $H$ is \defn{graceful} if such a labelling exists. The Graceful tree conjecture asserts that all trees are graceful. Rosa~\cite{rosa:67} showed that this would imply the aforementioned Ringel--Kotzig conjecture. 
Despite extensive research, this conjecture remains wide open. Adamaszek, Allen, Grosu and Hladk\'y~\cite{AAGH:16} recently proved that almost all trees are almost graceful. 

Harmonious labellings were introduced by Graham and Sloane~\cite{GS:80} in~1980. Given a graph~$H$ and an abelian group~$\Gamma$, a \defn{$\Gamma$-harmonious labelling of~$H$} is an injective map $f\colon V(H)\to \Gamma$ such that the induced edge labels $f(x)+f(y)$, $xy\in E(H)$, are pairwise distinct, and $H$ is \defn{$\Gamma$-harmonious} if such a labelling exists. Graham and Sloane asked which graphs $H$ are $\bZ_{e(H)}$-harmonious. Note that this necessitates that $|V(H)|\le e(H)$. 
In the special case when $H$ is a tree on $n$~vertices, they conjectured that there exists an injective map $f\colon V(H)\to [n]$ such that the induced edge labels $f(x)+f(y)$, $xy\in E(H)$, are pairwise distinct modulo~$n-1$. 
\.{Z}ak~\cite{zak:09} proposed a weakening of this. He conjectured that every tree on $n-o(n)$ vertices is $\bZ_{n}$-harmonious. Montgomery, Pokrovskiy and Sudakov~\cite{MPS:18} proved \.{Z}ak's conjecture as a corollary of Theorem~\ref{thm:MPS}. Using our Theorem~\ref{thm:main simplified}, we can deduce a similar statement for general bounded degree graphs.

\begin{theorem}\label{thm:harmonious labellings}
	For all $\Delta\in \bN$, there exist $\eps>0$ and $n_0\in\mathbb{N}$ such that the following holds for all $n\geq n_0$.
	Suppose $H$ is a graph with at most $n$~vertices, at most $(1-\eps )n$ edges and $\Delta(H)\leq\Delta$. Let $\Gamma$ be an abelian group of order~$n$.
	Then $H$ is $\Gamma$-harmonious.
\end{theorem}
\proof
Consider the complete graph $K_\Gamma$ on~$\Gamma$. Define the edge-colouring $c\colon E(K_\Gamma)\to \Gamma$ by setting $c(ij)=i+j$, and note that $c$ is proper and thus $n/2$-bounded.\COMMENT{$n/2\le (1-2\eps)\frac{\binom{n}{2}}{e(H)}$}
Hence, by Theorem~\ref{thm:main simplified}, $K_{\Gamma}$ contains a rainbow copy of~$H$, which corresponds to a $\Gamma$-harmonious labelling of~$H$.
\endproof

\section{Proof overview}\label{sec:overview}

In the literature, there are two common approaches for proving blow-up lemmas. The original approach of Koml\'os, S\'ark\"ozy and Szemer\'edi consists of a randomised sequential embedding algorithm, which embeds the bulk of the vertices one-by-one, choosing each time a random image from all available ones. This strategy has also been used in~\cite{ABHKP:16,BKTW:15,csaba:07,keevash:11}. 

Shortly after the appearance of the blow-up lemma, R\"odl and Ruci\'nski~\cite{RR:99} developed an alternative proof, where instead of embedding vertices one-by-one, the algorithm consists of only a constant number of steps. In the $i$th step, the whole cluster $X_i$ is embedded into~$V_i$. The desired bijection is obtained as a perfect matching within a `candidacy graph'~$A_i$, which is an auxiliary bipartite graph between $X_i$ and $V_i$ where $xv\in E(A_i)$ only if $v$ is still a suitable image for~$x$. Although these candidacy graphs (of clusters not yet embedded) become sparser after each step, R\"odl and Ruci\'nski were able to show that one can maintain their super-regularity throughout the procedure. 
This approach was also employed in~\cite{KKOT:19} to prove a blow-up lemma for approximate decompositions, and in~\cite{GJ:18} to prove a rainbow blow-up lemma for $o(n)$-bounded colourings, and also underpins our proof here.

For simplicity, we consider here the following setup. Suppose $V(H)$ is partitioned into independent sets $X_1,X_2,X_3$ of size $n$ and $H$ consists of a perfect matching between $X_1$ and $X_2$, and a perfect matching between $X_2$ and~$X_3$.

 Suppose that we have already found an embedding $\phi_1\colon X_1\to V_1$, and next we want to embed $X_2$ into~$V_2$. We define the bipartite graph $A_2$ between $X_2$ and $V_2$ by adding the edge $xv$ if $\phi_1(y)v\in E(G)$, where $y$ is the $H$-neighbour of $x$ in~$X_1$. Now, the aim is to find a perfect matching $\sigma$ in~$A_2$. Note that any such perfect matching yields a valid embedding of $H[X_1,X_2]$ into $G[V_1,V_2]$. Moreover, if we aim to find a rainbow embedding, this can be achieved as follows. For each $xv\in E(A_2)$, we colour $xv$ with the colour of $\phi_1(y)v$. Observe that if $\sigma$ is rainbow, then the embedding of $H[X_1,X_2]$ into $G[V_1,V_2]$ will be rainbow, too. Let us assume that $A_2$ is super-regular. It is well known that $A_2$ then has a perfect matching. One key ingredient in~\cite{GJ:18} was to combine this fact with a recent result of Coulson and Perarnau~\cite{CP:17}, based on the switching method, to even find a rainbow perfect matching. Unfortunately, the switching method relies upon the fact that the given colouring is $o(n)$-bounded, and is thus not applicable in the present setting. 
There are two key insights that will allow us to deal with almost optimally bounded colourings.

First, note that given a proper colouring of a graph~$G$, if we take a random subset $U$ of size $\mu |G|$, then with high probability, the colouring induced on $U$ will be $(1+o(1))\mu |U|$-bounded, and thus the rainbow blow-up lemma from~\cite{GJ:18} is applicable (on $U$). This gives hope to combine this with an `approximate result' on $V(G)\sm U$ to obtain the desired embedding. Such a combination of techniques has already been successfully used in~\cite{KKKO:18}.
In our simplified discussion, let us thus assume we do not need to find a perfect rainbow matching~$\sigma$, but would be content if $\sigma$ is almost perfect.

This leads us to the second main ingredient of our proof---matchings in hypergraphs. Given our candidacy graph $A_2$ and its (auxiliary) colouring $c_2\colon E(A_2)\to C_2$, we define a hypergraph~$\cH$ on $X_2\cup V_2 \cup C_2$ where for every edge $e\in E(A_2)$, we add the hyperedge $e\cup \Set{c(e)}$ to~$\cH$. A simple but crucial observation is that there is a one-to-one correspondence between matchings in $\cH$ and rainbow matchings in $A_2$. In particular, a matching $\cM$ in $\cH$ that covers almost all vertices of $X_2\cup V_2$ would translate into our desired almost perfect rainbow matching $\sigma$ in~$A_2$. Here, we can make use of the rich theory of matchings in hypergraphs with small codegrees, which we will discuss in more detail in Section~\ref{sec:hypermatchings}. At this point, we remark that since $A_2$ is super-regular, all vertices of $X_2\cup V_2$ have roughly the same degree in~$\cH$, and if the degrees of the colours are not larger (that is, the colouring is appropriately bounded), this will suffice to find the desired matching in~$\cH$.

Moreover, note that we assumed that $A_2$ is super-regular and its colouring is appropriately bounded. After embedding $X_2$ according to~$\sigma$, we have to \emph{update} the candidacy graph~$A_3$ as we updated $A_2$ after embedding~$X_1$. Of course, whether $A_3$ will be super-regular and its colouring appropriately bounded depends heavily on~$\sigma$. For the embedding not to get stuck, we need to find in $A_2$ not just \emph{any} $\sigma$, but a \emph{good} one.
To achieve this, we make use of a general hypergraph matching theorem (Theorem~\ref{thm:hypermatching}) proved recently by the authors which guarantees a matching~$\cM$ in $\cH$ that is in many ways `random-like'. This will allow us to find an embedding~$\sigma$ for which the updated candidacy graph $A_3$ will have the desired properties.

\section{Preliminaries}\label{sec:preliminaries}

\subsection{Notation}
For $k\in\mathbb{N}$, we write $[k]_0:=[k]\cup\{0\}=\{0,1,\ldots,k\}$, where $[0]=\emptyset$.
For a finite set $S$ and $i\in \mathbb{N}$,
we write $\binom{S}{i}$ for the set of all subsets of $S$ of size $i$
and $2^S$ for the powerset of $S$. 
For a set $\Set{i,j}$, we sometimes simply write $ij$.
For $a,b,c\in \mathbb{R}$,
we write $a=b\pm c$ whenever $a\in [b-c,b+c]$.
For $a,b,c\in (0,1]$,
we sometimes write $a\ll b \ll c$ in our statements meaning that there are increasing functions $f,g:(0,1]\to (0,1]$
such that whenever $a\leq f(b)$ and $b \leq g(c)$,
then the subsequent result holds.

We only consider finite, simple and undirected graphs.
For a graph~$G$, we let~$V(G)$ and~$E(G)$ denote the vertex set and edge set, respectively. 
We say $u\in V(G)$ is a $G$-neighbour of $v\in V(G)$ if $uv\in E(G)$.
As usual, $\Delta(G)$ denotes the maximum degree of~$G$.
For $u,v\in V(G)$, let $N_G(u,v):=N_G(u)\cap N_G(v)$ denote the common neighbourhood of $u$ and~$v$. For a set $S\In V(G)$, let $N(S):=\bigcup_{v\in S}N_G(v)$.
For disjoint subsets $A,B\In V(G)$, 
let $G[A,B]$ denote the bipartite subgraph of $G$ between $A$ and $B$
and $G[A]$ the subgraph in $G$ induced by $A$.
Let $e(G)$ be the number of edges of $G$ and let $e_G(A,B)$ denote the number of edges of $G[A,B]$.
We let $G^2$ denote the square of $G$, that is, the graph obtained from~$G$ by adding edges between vertices which have a common neighbour in~$G$.
A subset $X\In V(G)$ is $2$-independent if it is independent in~$G^2$.

Let $G$ be a graph.
Given a set $C$, 
a function $c\colon E(G)\to 2^C$ is called an \emph{edge set colouring} of~$G$. A colour $\alpha \in C$ \emph{appears} on an edge $e$ if $\alpha\in c(e)$. 
We define the \defn{codegree} of $c$ as the maximum number of edges on which any two fixed colours appear together.
For a colour $\alpha \in C$, a vertex $v\in V(G)$, and disjoint sets $A,B\subseteq V(G)$, we define
\begin{itemize}
\item  $\dg_G^\alpha(v):=|\{u\in N_G(v)\colon \text{$\alpha$ appears on $uv$} \}|$;
\item $e^\alpha_G(A,B):=|\{ab\in E(G)\colon a\in A, b\in B, \text{ and $\alpha$ appears on $ab$} \}|$;
\item $e^\alpha(G):=|\{e\in E(G)\colon \text{$\alpha$ appears on $e$} \}|$.
\end{itemize}
We say that
\begin{itemize}
\item $c$ is \defn{(globally) $k$-bounded} if each colour appears on at most $k$ edges;
\item $c$ is \defn{locally $\Lambda$-bounded} if each colour class has maximum degree at most~$\Lambda$.
\end{itemize}
Given a partition $(V_i)_{i\in[r]}$ of $V(G)$, we say that $c$ is \defn{colour-split with respect to $(V_i)_{i\in[r]}$} if for all $e,f\in E(G)$ we have $c(e)\cap c(f)=\emptyset$ whenever $e\in E(G[V_i,V_j])$ and $f\notin E(G[V_i,V_j])$.
If the partition is clear from the context, we just say that $c$ is colour-split.
We call a subgraph $G'$ of $G$ \emph{rainbow} if all the edges in $G'$ have pairwise disjoint colour sets. 

\subsection{Probabilistic tools} In this section, we state a well-known Chernoff-type bound and McDiarmid's inequality. These will be the main tools to establish concentration of a random variable for the large deviation results we need.\COMMENT{We do not use the Chernoff bound explicitly. Just something like ``with a Chernoff bound, one can see..."}

\begin{theorem}[Chernoff's bound, see~\cite{JLR:00}]\label{thm:general chernoff simple}
Suppose $X_1,\dots,X_m$ are independent Bernoulli random variables. Let $X:=\sum_{i=1}^m X_i$. Then, for all $t\ge 0$, $$\prob{|X-\expn{X}|\ge t} \le 2 \exp\left({-\frac{t^2}{2(\expn{X} + t/3)}}\right).$$
\end{theorem}
\COMMENT{For $0\le \eps\le 3/2$, this implies (with $t=\eps\expn{X}$) that $\prob{X\neq (1\pm \eps)\expn{X}} \le 2 \exp\left({-\eps^2\expn{X}/3}\right).$ If $t\ge 7\expn{X}$, we obtain $\prob{X \ge t} \le 2 \exp(-t).$}

\begin{theorem}[McDiarmid's inequality, see~{\cite[Lemma~1.2]{mcdiarmid:89}}] \label{thm:McDiarmid}
Suppose $X_1,\dots,X_m$ are independent Bernoulli random variables and suppose $b_1,\dots,b_m\in [0,B]$.
Suppose $X$ is a real-valued random variable determined by $X_1,\dots,X_m$ such that changing the outcome of $X_i$ changes $X$ by at most $b_i$ for all $i\in [m]$.
Then, for all $t>0$, we have 
$$\prob{|X-\expn{X}|\ge t} 
\le 2 \exp\left({-\frac{2t^2}{\sum_{i=1}^m b_i^2}}\right)
\le 2 \exp\left({-\frac{2t^2}{B\sum_{i=1}^m b_i}}\right).$$
\end{theorem}

\subsection{Pseudorandom hypergraph matchings}\label{sec:hypermatchings}

As sketched in Section~\ref{sec:overview}, we use hypergraph matchings to model rainbow embeddings. 
In this section, we introduce a theorem from~\cite{EGJ:19a} on `pseudorandom' hypergraph matchings (Theorem~\ref{thm:hypermatching}) which will play an important role in Section~\ref{sec:4graphs new2}.

Following the seminal result of R\"odl~\cite{rodl:85} on approximate Steiner systems, Pippenger observed that any almost regular uniform hypergraph with small codegrees has an almost perfect matching.
In~\cite{EGJ:19a}, the authors proved a tool which allows to obtain `pseudorandom' matchings in this setting. 
To make this more precise, we define 
for a hypergraph $\cH$ and vertices $u,v\in V(\cH)$ the degree $\dg_{\cH}(v):=|\{e\in E(\cH)\colon v\in e \}|$ and codegree $\dg_{\cH}(uv):=|\{e\in E(\cH)\colon \{u,v \}\subseteq e \}|$.
We define
\begin{align*}
\Delta(\cH):=\max_{v\in V(\cH)}\dg_{\cH}(v) \text{~~~and~~~}
\Delta^c(\cH):=\max_{u\neq v\in V(\cH)} \dg_{\cH}(uv)
\end{align*}
to be the \defn{maximum degree} and \defn{maximum codegree} of~$\cH$, respectively.
A \emph{matching} in $\cH$ is a collection of disjoint edges.

Suppose for simplicity that we are given a $D$-regular hypergraph and want to find an (almost) perfect matching~$\cM$. Moreover, we wish $\cM$ to be `pseudorandom', that is, to have certain properties that we expect from an idealized random matching. Heuristically, we may expect that every edge of $\cH$ is in a random perfect matching with probability~$1/D$. Thus, given a subset $U\In V(\cH)$, we expect $|E(\cH[U])|/D$ matching edges inside~$U$, and we may require similar statistics for matching edges crossing certain vertex sets. This can be formalized in a quite general way.
Given a set $X$, a \defn{weight function on~$X$} is a function $\omega\colon X\to \bR_{\ge 0}$.
For a subset $X'\In X$, we define $\omega(X'):=\sum_{x\in X'}\omega(x)$. If $\omega$ is a weight function on $E(\cH)$, the above heuristic would imply that we expect $\omega(\cM)\approx \omega(E(\cH))/D$.
The following theorem asserts that a hypergraph with small codegrees has a matching that is pseudorandom in this sense.

\begin{theorem}[\cite{EGJ:19a}]\label{thm:hypermatching}
Suppose $1/\Delta\ll\delta,1/r$ and $r\in \bN$ with $r\ge 2$, and let $\eps:=\delta/50r^2$. 
Let $\cH$ be an $r$-uniform hypergraph with $\Delta(\cH)\leq \Delta$ and $\Delta^c(\cH)\le \Delta^{1-\delta}$ as well as $e(\cH)\leq \exp(\Delta^{\eps^2})$. 
Suppose that $\cW$ is a set of at most $\exp(\Delta^{\eps^2})$ weight functions on~$E(\cH)$.
Then, there exists a matching $\cM$ in~$\cH$ such that $\omega(\cM)=(1\pm \Delta^{-\eps}) \omega(E(\cH))/\Delta$ for all $\omega \in \cW$ with $\omega(E(\cH))\ge \max_{e\in E(\cH)}\omega(e)\Delta^{1+\delta}$.
\end{theorem}

We refer the interested reader to \cite{EGJ:19a} for more information on preceding results and further variants and applications of Theorem~\ref{thm:hypermatching}.

\subsection{Regularity} 
Given $\eps>0$ and $d\in[0,1]$, a bipartite graph $G$ with vertex classes $(V_1,V_2)$ is called \defn{$(\eps,d)$-regular} if for all pairs $S\In V_1$ and $T\In V_2$ with $|S|\ge \eps|V_1|$, $|T|\ge \eps |V_2|$,
we have $d_G(S,T)=d\pm\eps$. %where $d_G(S,T):= {e_G(S,T)}/{(|S||T|)}$.
%We simply say that $G$ is \defn{$\eps$-regular} if $G$ is $(\eps,d_G(V_1,V_2))$-regular.
The following is one of the fundamental properties of $\eps$-regularity.
\begin{fact} \label{fact:regularity}
Let $G$ be an $(\eps,d)$-regular bipartite graph with partition $(A,B)$, and let $Y\In B$ with $|Y|\ge \eps|B|$. Then all but at most $2\eps|A|$ vertices of $A$ have $(d\pm \eps)|Y|$ neighbours in~$Y$.  
\end{fact}

We will also often use the fact that super-regularity is robust with respect to small vertex and edge deletions.
\begin{fact}\label{fact:regularity robust}
Suppose $1/n \ll \eps \ll d$.
Let $G$ be an $(\eps,d)$-super-regular bipartite graph with partition~$(A,B)$, where $\eps^{1/6} n\le |A|,|B|\le n$. If $\Delta(H)\le \eps n$ and $X\In A\cup B$ with $|X|\le \eps n$, then $G[A\sm X,B\sm X]-E(H)$ is $(\eps^{1/3},d)$-super-regular.\COMMENT{Proof: The number of edges between two vertex sets $Z_1$ and $Z_2$ is at least 
$(d-\eps)|Z_1||Z_2|- (\eps n + \eps n)(|Z_1|+|Z_2|)\geq |Z_1||Z_2|(d- \eps - 4\eps^{1/2})\geq |Z_1||Z_2|(d-\eps^{1/3})$ whenever $|Z_i|\geq \eps^{1/2}n= \eps^{1/3} \cdot \eps^{1/6}n$.}
\end{fact}

The following is essentially a result from~\cite{DLR:95}.
(In~\cite{DLR:95} it is proved in the case when $|A|=|B|$ with $16\epsilon^{1/5}$ instead of $\epsilon^{1/6}$.
The version stated below can be easily derived from this.)

\begin{theorem}\label{thm: almost quasirandom}
Suppose $1/n\ll\eps \ll \gamma,d $.
Suppose $G$ is a bipartite graph with vertex partition $(A,B)$ such that $|A|=n$, $\gamma n \leq |B|\leq \gamma^{-1}n$ and at least $(1-5\eps)n^2/2$ pairs $u,v\in A$
satisfy $\dg_G(u),\dg_G(v)\geq (d- \eps)|B|$ and $|N_G(u,v)|\leq (d+ \eps)^2|B|$.
Then $G$ is $(\eps^{1/6},d)$-regular. 
\end{theorem}

\subsection{Another rainbow blow-up lemma}
Our final tool is the following special case of the rainbow blow-up lemma from~\cite{GJ:18} for $o(n)$-bounded colourings. Even though the global boundedness condition is more restrictive there, it is still applicable on a random subset of vertices (see the discussion in Section~\ref{sec:overview}). 
As such, it is the main tool in our proof to turn a partial rainbow embedding into a complete one.

We say that $(H,G,(X_i)_{i\in[r]_0}, (V_i)_{i\in[r]_0})$ is an \defn{$(\eps,d)$-super-regular blow-up instance with exceptional sets $(X_0,V_0)$} if $X_0$ is an independent set in $H$, $|V_0|=|X_0|$ and $(H-X_0,G-V_0, (X_i)_{i\in[r]}, (V_i)_{i\in[r]})$ is an $(\eps,d)$-super-regular blow-up instance.
We call graphs $(A_i)_{i\in[r]}$ \emph{candidacy graphs} if $A_i$ is a bipartite graph with partition $(X_i,V_i)$ for all $i\in[r]$.

\begin{lemma}[{\cite[Lemma~5.2]{GJ:18}}]\label{lem:blow up matchings}
Suppose $1/n\ll\eps,\mu\ll d,1/r,1/\Delta$.
Let $(H,G,(X_i)_{i\in[r]_0},(V_i)_{i\in[r]_0})$ be an $(\eps,d_G)$-super-regular blow-up instance with exceptional sets $(X_0,V_0)$ and $(\eps,d_A)$-super-regular candidacy graphs $(A_i)_{i\in[r]}$, where $d_G,d_A\ge d$. Assume further that
%$\Delta(H),\Delta(R)\le \Delta$ and $|V_i|=(1\pm \eps)n/r$ for all $i\in[r]$.
\begin{enumerate}[label={\rm (\roman*)}]
	\item $\Delta(H) \leq \Delta$;
	\item\label{lem:cond2} $|V_i|=n$ for all $i\in[r]$;\COMMENT{For the absorption, $V_i$ will be the absorber sets of size $|V_i|=(\mu\pm\eps)n$. So in fact, in the absorption step, we might add even more vertices to the absorber than just the leftover vertices, in order to guarantee that every absorber set has equal size, say $(\mu+2\eps_r)n=:n'$. That is, we even tear off already embedded vertices and put them in the absorber.}
	\item $H[X_i,X_j]$ is a matching for all $ij\in \binom{[r]}{2}$.
\end{enumerate}
Let $c\colon E(G)\to C$ be a $\mu n$-bounded edge-colouring of $G$.
Suppose a bijection $\psi_0 \colon X_0 \to V_0$ is given such that 
\begin{enumerate}[label={\rm (\roman*)}]
\setcounter{enumi}{3}
\item for all $x\in X_0$, $i\in[r]$ and $x_i\in N_H(x)\cap X_i$,
we have $N_{A_i}(x_i)\subseteq N_G(\psi_0(x))$;
\item for all $i\in [r]$, $x\in X_i$, $v\in N_{A_i}(x)$ and distinct $x_0,x_0'\in N_H(x)\cap X_0$, we have $c(\psi_0(x_0)v)\neq c(\psi_0(x_0')v)$.\COMMENT{before: $c(\psi_0(x_0)v)\cap c(\psi_0(x_0')v)=\emptyset$, which is from the set colouring version, but have simplified here to normal colourings}
\end{enumerate}
Then there exists a rainbow embedding $\psi$ of $H$ into $G$ which extends $\psi_0$
such that $\psi(x)\in N_{A_i}(x)$ for all $i\in [r]$ and $x\in X_i$.
\end{lemma}

\COMMENT{Lemma 5.2. In the original statement we have $1/n\ll\mu\ll\eps\ll d,1/r$. I changed this to $1/n\ll\eps,\mu\ll d,1/r$, compare Lemma 5.1.}

%%%%%%%%%%%%%%%%%%%%%%%%%%%%%%%%%%%%%%%%%%%%%%%%%%%%%%%%%%%%%%%
\section{Colour splitting}\label{sec:colour splitting}
The goal of this section is to provide some useful lemmas to refine the partitions of a blow-up instance and split the colours into groups in order to obtain better control for the rainbow embedding.

The first lemma will guarantee that with high probability the resulting graph is still super-regular when we randomly split colours in order to obtain a colour-split colouring.

\begin{lemma}\label{lem:regularity colour splitting}
Let $1/n\ll\eps\ll\eps'\ll\gamma,d,1/\Lambda$. 
Suppose $G$ is an $(\eps,d)$-super-regular graph with vertex partition $(A,B)$ such that $|A|,|B|=(1\pm\eps)n$, and $c\colon E(G)\to C$ is a locally $\Lambda$-bounded edge-colouring of $G$.
Suppose $\{Y_\alpha\colon \alpha\in C \}\cup\{Z_e\colon e\in E(G) \}$ is a set of mutually independent Bernoulli random variables such that $\prob{Y_{c(e)}+Z_e=2}=\gamma$ for every $e\in E(G)$.\COMMENT{The variables $Y_\alpha$ and $Z_e$ describe a two-phase random process, where first each colour is kept independently at random according to $Y_\alpha$, and second each edge is kept independently at random according to $Z_e$.}
Suppose $G'$ is the random spanning subgraph of $G$ where $e\in E(G)$ belongs to $E(G')$ whenever $Y_{c(e)}+Z_e=2$.
Then $G'$ is $(\eps',\gamma d)$-super-regular with probability at least $1-1/n^{10}$.
\end{lemma}

\begin{proof}
We call a pair of distinct vertices $u,v\in A$ \emph{good} if $|N_G(u,v)|=(d\pm\eps)^2|B|$, and $|\{w\in N_G(u,v)\colon c(uw)= c(vw) \}|\leq\eps|B|$.
We first claim that almost all pairs are good.
\begin{claim}
There are at least $(1-7\eps)|A|^2/2$ good pairs $u,v\in A$.
\end{claim}

\claimproof
Since $G$ is $(\eps,d)$-super-regular, at most $2\eps|A|^2$ pairs $u,v\in A$ do not satisfy $|N_G(u,v)|=(d\pm\eps)^2|B|$ by Fact~\ref{fact:regularity}.\COMMENT{$|A|$ choices for $u$. Let $Y=N_G(u)$. $|Y|=(d\pm \eps)|B|$. All but $2\eps|A|$ vertices $v$ have $(d\pm \eps)|Y|=(d\pm\eps)^2|B|$ neighbours in $Y$. Could divide by two since each bad pair is counted twice...}

We claim that the number of pairs $u,v\in A$ with $|\{w\in N_G(u,v)\colon c(uw)= c(vw) \}|\ge \eps|B|$ is at most $\eps |A|^2$.
For this, we first count the number of monochromatic paths of length~$2$ in $G$ with both ends in $A$. Each vertex $w\in B$ is 
contained in $\sum_{\alpha\in C}\binom{\dg^\alpha_G(w)}{2}$ monochromatic paths $uwv$ in $G$.
Since $\dg^\alpha_G(w)\leq\Lambda$ for every colour $\alpha\in C$ and $\sum_{\alpha\in C}\dg^\alpha_G(w)\leq |A|$, we have
\begin{align*}
\sum_{\alpha\in C}\binom{\dg^\alpha_G(w)}{2}
\leq \sum_{\alpha\in C}\dg^\alpha_G(w)^2
\leq \Lambda |A|.
\end{align*}
Hence, there are at most $\Lambda |A||B|$ monochromatic paths of length 2 in $G$ with both ends in~$A$. 
This implies that the number of pairs $u,v\in A$ with $|\{w\in N_G(u,v)\colon c(uw)= c(vw) \}|\ge \eps|B|$ is at most $$\frac{\Lambda|A||B|}{\eps |B|}\leq \eps |A|^2.$$ 

Thus, there are at least 
$\binom{|A|}{2}-3\eps|A|^2\geq (1-7\eps)|A|^2/2$ 
good pairs $u,v\in A$.
\endclaimproof

We fix a vertex $x\in A\cup B$ and a good pair of vertices $u,v\in A$. 
Let $X_x:=\dg_{G'}(x)$ and $X_{u,v}:=|N_{G'}(u,v)|$.
Clearly, $X_x$ and $X_{u,v}$ are determined by $\{Y_\alpha\colon\alpha\in C\}\cup\{Z_e\colon e\in E(G) \}$. Note that if $w\in N_G(u,v)$ satisfies $c(uw)\neq c(vw)$, then $\prob{w\in N_{G'}(u,v)}=\gamma^2$. Thus, we have
\begin{align}\label{eq:expectations}
\expn{X_x}=\gamma \dg_G(x) = \gamma d n \pm 3\eps n \text{~~~and~~~} \expn{X_{u,v}}=\gamma^2 d^2 n \pm 10 \eps n.
\end{align}
\COMMENT{$\gamma \dg_G(x)=\gamma (d\pm \eps)(1\pm \eps)n$, $\expn{X_{u,v}} = \gamma^2|\{w\in N_G(u,v)\colon c(uw)\neq c(vw) \}| \pm |\{w\in N_G(u,v)\colon c(uw)= c(vw) \}| = \gamma^2 (|N_G(u,v)|\pm \eps |B|) \pm \eps |B| = \gamma^2 (d\pm \eps)^2 |B| \pm 2\eps |B| = \gamma^2 (d^2\pm 3\eps)|B| \pm 2\eps |B|$}
For all $\alpha\in C$ and $e\in E(G)$, let $b_\alpha$ and $b_e$ be minimally chosen such that changing the outcome of $Y_\alpha$ changes $X_x$ by at most $b_\alpha$, and changing the outcome of $Z_e$ changes $X_x$ by at most~$b_e$. Note that
\begin{align*}
\sum_{\alpha\in C}b_\alpha+\sum_{e\in E(G)}b_e \leq 2\dg_G(x)\le 3n.
\end{align*}
%Since $x$ has at most $(1+\eps)n$ neighbours in $G$, we have\COMMENT{Since $u$ has at most $(1+\eps)n$ neighbours in $G$, we have $\sum_{\alpha\in C}b_\alpha\leq (1+\eps)n$ and similarly $\sum_{e\in E(G)}b_e\leq (1+\eps)n.$}
Moreover, we clearly have $b_e\leq 1$, and since the colouring $c$ is locally $\Lambda$-bounded, $b_\alpha\leq\Lambda$.
Using McDiarmid's inequality (Theorem~\ref{thm:McDiarmid}), we obtain that 
\begin{align}\label{eq:McDiarmid X_x}
\prob{\left|X_x-\expn{X_x}\right|>\eps n}\leq 2\exp\left(-\frac{\eps^2n^2}{\Lambda \cdot 3n}\right)
<\frac{1}{n^{20}}.
\end{align}
With similar arguments one can show that
\begin{align}\label{eq:McDiarmid X_u,v}
\prob{\left|X_{u,v}-\expn{X_{u,v}}\right|>\eps n}
<\frac{1}{n^{20}}.
\end{align}
A union bound over all $x\in A\cup B$ and all good pairs $u,v\in A$ yields together with~\eqref{eq:expectations}, \eqref{eq:McDiarmid X_x} and~\eqref{eq:McDiarmid X_u,v} that with probability at least $1-1/n^{10}$, we have $\deg_{G'}(x)=\gamma d n \pm 4\eps n$  for all $x\in A\cup B$, and
$|N_{G'}(u,v)|=\gamma^2 d^2 n \pm 11 \eps n$ for all good pairs $u,v\in A$.
%For a good pair $u,v\in A$, it holds $|\{w\in N_G(u,v)\colon c(uw)= c(vw) \}|\leq\eps|B|$, and thus by~\eqref{eq:u,v colours}, we conclude that 
%\begin{align}\label{eq:codegree}
%|N_{G'}(u,v)|=(\gamma d\pm 3\sqrt{\eps})^2|B|.
%\end{align}
Given that, Theorem~\ref{thm: almost quasirandom} implies that $G'$ is $(\eps',\gamma d)$-super-regular.
\end{proof}

The next lemma states that we can split the colours of the host graph $G$ into groups and obtain a subgraph $G'$ which is still super-regular, and whose colouring is colour-split and appropriately bounded.

\begin{lemma}\label{lem:colour splitting}
Let $1/n\ll\eps\ll\eps'\ll d'\ll\gamma\ll d,1/\Lambda,1/r,1/\Delta$. 
Suppose $(H,G, (X_i)_{i\in[r]}, (V_i)_{i\in[r]})$ is an $(\eps,d)$-super-regular blow-up instance.
Assume further that 
\begin{enumerate}[label={\rm (\roman*)}]
	\item\label{Delta(H)} $\Delta(H)\leq \Delta$ and $e_H(X_i,X_j)\geq\gamma^2n$ for all $ij\in \binom{[r]}{2}$;
	\item\label{size V_i} $|V_i|=(1\pm\eps)n$ for all $i\in[r]$;
	\item\label{sum boundedness} $c\colon E(G)\to C$ is locally $\Lambda$-bounded and the following holds for all $\alpha\in C$:
	$$\sum_{ij\in \binom{[r]}{2}}e^\alpha_G(V_i,V_j)e_H(X_i,X_j)\leq (1-\gamma)dn^2.$$
\end{enumerate}
Then there exists a spanning subgraph $G'$ of $G$ such that 
\begin{enumerate}[label={\rm (\alph*)}]
\item\label{G' blow-up instance} $(H,G', (X_i)_{i\in[r]}, (V_i)_{i\in[r]})$ is an $(\eps',d')$-super-regular blow-up instance;
\item $c$ restricted to $G'$ is colour-split;
\item\label{c G' bounded} $c$ restricted to $G'[V_i,V_j]$ is $(1-\frac{\gamma}{2})\frac{e_{G'}(V_i,V_j)}{e_H(X_i,X_j)}$-bounded for all $ij\in\binom{[r]}{2}$.
\end{enumerate}
\end{lemma}

\begin{proof}
Let $\heps$ be such that $\eps\ll\heps\ll\eps'$.
The proof proceeds in three steps, where we iteratively define spanning subgraphs $G_3\subseteq G_2\subseteq G_1\subseteq G$ such that $G_3$ satisfies the required properties of $G'$ in the statement.

In the first step we suitably sparsify each bipartite subgraph $G[V_i,V_j]$.
For every $ij\in \binom{[r]}{2}$, let
\begin{align}\label{eq:def p_ij}
p_{ij}:=\frac{e_H(X_i,X_j)}{2\Delta n}.
\end{align}
Note that $\gamma^2/(2\Delta)\leq p_{ij}\leq 1$ since $\gamma^2 n\le e_H(X_i,X_j) \le \Delta |X_i|\le 2\Delta n$.
For every $ij\in \binom{[r]}{2}$, we keep each edge of $G[V_i,V_j]$ independently at random with probability $p_{ij}$ and denote the resulting graph by~$G_1[V_i,V_j]$.
A simple application of Chernoff's inequality together with a union bound yields the following claim.
\begin{claim}\label{claim G_1}
The following properties hold simultaneously with probability at least $1-1/n$ for every $ij\in \binom{[r]}{2}$.
\begin{enumerate}[label={\rm (C1.\arabic*)}]
\item\label{G_1 super-regular} $G_1[V_i,V_j]$ is $(2\eps,p_{ij}d)$-super-regular;
%\item\label{e_G_1} $e_{G_1}(V_i,V_j)=(1\pm\eps)e_{G}(V_i,V_j)p_{ij}$;
\item\label{e_G_1 alpha} $e^\alpha_{G_1}(V_i,V_j)\leq e^\alpha_{G}(V_i,V_j)p_{ij}+ \eps n$ for every colour $\alpha\in C$.
\end{enumerate}
\end{claim}

\COMMENT{We can apply Chernoff's inequality to the set $\{Z_e\colon e\in E(G[V_i,V_j])\}$ of independent Bernoulli random variables where $Z_e$ indicates whether $e\in E(G[V_i,V_j])$ is present in $G_1[V_i,V_j]$. 
It holds $$\dg_{G_1[V_i,V_j]}(v)=\sum_{v\in e} Z_e$$
and 
$$e_{G_1}(S,T)=\sum_{st\in E(G),s\in S,t\in T} Z_{st},$$
for $S\subseteq V_i$, $T\subseteq V_j$. 
Hence, Chernoff's inequality yields the required concentration since the random variables $Z_e$ are independent. This implies the claim.
}

Hence, by Claim~\ref{claim G_1}, we may assume that $G_1$ is a spanning subgraph of $G$ such that properties~\ref{G_1 super-regular}--\ref{e_G_1 alpha} hold.
For every colour $\alpha\in C$, we obtain that
\begin{align}\label{eq:sum boundedness G_1}
&\sum_{ij\in \binom{[r]}{2}}e^\alpha_{G_1}(V_i,V_j)\frac{e_H(X_i,X_j)}{e_{G_1}(V_i,V_j)}
\stackrel{\ref{G_1 super-regular},\ref{e_G_1 alpha}}{\leq}
\sum_{ij\in \binom{[r]}{2}}(e^\alpha_{G}(V_i,V_j)p_{ij}+\eps n)\frac{e_H(X_i,X_j)}{(1-\sqrt{\eps}) p_{ij} d n^2}
\notag \\
&\stackrel{\hphantom{\ref{sum boundedness},\eqref{eq:def p_ij}}}{\leq}(1+2\sqrt{\eps})\sum_{ij\in \binom{[r]}{2}}e^\alpha_{G}(V_i,V_j)\frac{e_H(X_i,X_j)}{d n^2}
+\eps n \sum_{ij\in \binom{[r]}{2}}\frac{2e_H(X_i,X_j)}{p_{ij} d n^2}
\notag \\
&\stackrel{\ref{sum boundedness},\eqref{eq:def p_ij}}{\leq}
(1+2\sqrt{\eps})(1-\gamma)+\eps n \binom{r}{2}\frac{4\Delta n}{dn^2}
\leq 1-\gamma+3\sqrt{\eps} \le
1-\frac{3\gamma}{4}.
\end{align}

Note that~\eqref{eq:def p_ij} and~\ref{G_1 super-regular} imply that 
\begin{align}\label{eq:sparsify G_1 holds}
\frac{e_{G_1}(V_i,V_j)}{e_H(X_i,X_j)}= \frac{(d\pm\sqrt{\eps})n}{2\Delta}.
\end{align}
\COMMENT{$e_{G_1}(V_i,V_j)=(p_{ij}d\pm 2\eps)(1\pm \eps)^2 n^2=p_{ij}(d\pm \sqrt{\eps})n^2$}
Hence, for every colour $\alpha\in C$, we obtain
\begin{align}\label{eq:e^alpha(G_1)}
e^\alpha(G_1)
=\sum_{ij\in \binom{[r]}{2}}e^\alpha_{G_1}(V_i,V_j)
&\stackrel{\eqref{eq:sparsify G_1 holds}}{\leq} \frac{(d+\sqrt{\eps})n}{2\Delta} \sum_{ij\in \binom{[r]}{2}}e^\alpha_{G_1}(V_i,V_j)\frac{e_H(X_i,X_j)}{e_{G_1}(V_i,V_j)}
\notag\\
&\stackrel{\eqref{eq:sum boundedness G_1}}{\leq} \left(1-\frac{3\gamma}{4}\right)\frac{(d+\sqrt{\eps})n}{2\Delta}.
\end{align}

In the next step we define a random subgraph $G_2\subseteq G_1$. 
This will ensure that the final colouring is colour-split.
We choose $\tau\colon C\to\binom{[r]}{2}$ where each $\tau(\alpha)$ is chosen independently at random according to some probability distribution $(q_{ij}^\alpha)_{ij\in\binom{[r]}{2}}$, 
and for each $ij\in\binom{[r]}{2}$ and each edge $e$ of $G_1[V_i,V_j]$, let $Z_e$ be a Bernoulli random variable with parameter $\gamma^2/q_{ij}^{c(e)}$, all independent and independent of the choice of $\tau$.
Define $G_2$ by keeping each edge $e\in E_{G_1}(V_i,V_j)$ if $\tau(c(e))=ij$ and $Z_e=1$.
Hence,
\begin{align}\label{eq:edge survival}
\text{for all $e\in E(G_1)$, we have $\prob{e\in E(G_2)}
=\gamma^2$.}
\end{align}

We define $q_{ij}^\alpha$ as follows. For all $\alpha\in C$, let
\begin{align}\label{eq:I^alpha}
\cI^\alpha:=\left\{ij\in \binom{[r]}{2} \colon e^\alpha_{G_1}(V_i,V_j)>\frac{\gamma^2 e^\alpha(G_1)}{1-\binom{r}{2}\gamma^2}\right\} 
\text{~~~and~~~}
\cbI:=\binom{[r]}{2}\sm\cI^\alpha.
\end{align}
For $ij\in\cbI$, we set $q_{ij}^\alpha:=\gamma^2$.
For $ij\in\cI^\alpha$, we set\COMMENT{We want to make sure that~\eqref{eq:edge survival} holds, that is, each edge survives at least with probability $\gamma^2$. We have to do that in order to make sure that also bipartite pairs where many colours appear only sparely receive sufficiently many colours. 
Therefore, we set $q_{ij}^\alpha:=\gamma^2$ for $ij\in\cbI$.
For pairs $ij\in\cI^\alpha$ where colour $\alpha$ appears sufficiently many times we want to assign $\alpha$ to $ij$ according to the overall ratio of the appearance of $\alpha$ in $G$, that is, 
$$\frac{e^\alpha_{G_1}(V_i,V_j)}{\sum_{i'j'\in\cbI}e^\alpha_{G_1}(V_{i'},V_{j'})}.$$ In order to obtain a probability distribution $(q_{ij}^\alpha)_{ij\in\binom{[r]}{2}}$ we have to scale this and set $$q_{ij}^\alpha:=\left(1-\big|\cbI\big|\gamma^2\right)\frac{e^\alpha_{G_1}(V_i,V_j)}{\sum_{i'j'\in\cI^\alpha}e^\alpha_{G_1}(V_{i'},V_{j'})}$$ for $ij\in\cI^\alpha$.
}
\begin{align}\label{eq:p_ij^a}
q_{ij}^\alpha:=\left(1-\big|\cbI\big|\gamma^2\right)\frac{e^\alpha_{G_1}(V_i,V_j)}{\sum_{i'j'\in\cI^\alpha}e^\alpha_{G_1}(V_{i'},V_{j'})}.
\end{align}
Note that $\gamma^2\leq q_{ij}^\alpha\leq 1$ for all $ij\in\binom{[r]}{2}$, and $\sum_{ij\in\binom{[r]}{2}}q_{ij}^\alpha=1$.
\COMMENT{Note that if $ij\in \cbI$, then $e^\alpha_{G_1}(V_i,V_j)>0$, hence we do not divide by $0$.}

\begin{claim}\label{claim G_2}
The following properties hold simultaneously with probability at least $1-1/n$ for every $ij\in \binom{[r]}{2}$ and every colour $\alpha\in C$.
\begin{enumerate}[label={\rm (C2.\arabic*)}]
\item\label{G_2 super-regular} $G_2[V_i,V_j]$ is $(\heps,\gamma^2p_{ij}d)$-super-regular;
\item\label{nr colours G_2} $e^\alpha_{G_2}(V_i,V_j)\leq\frac{\gamma^2}{q_{ij}^\alpha}e^\alpha_{G_1}(V_i,V_j)+\eps n.$
\end{enumerate}
\end{claim}
\claimproof
For every $ij\in\binom{[r]}{2}$,
by~\eqref{eq:edge survival} and~\ref{G_1 super-regular}, Lemma~\ref{lem:regularity colour splitting} with $Y_\alpha=\mathbbm{1}_{\tau(\alpha)=ij}$ and $Z_e$ as defined above implies that~\ref{G_2 super-regular} holds with probability at least $1-1/n^5$.
 
In order to verify~\ref{nr colours G_2}, note that for $ij\in\binom{[r]}{2}$ the colour $\alpha$ appears in $G_2[V_i,V_j]$ only if $\tau(\alpha)=ij$.
Since we keep each $\alpha$-coloured edge independently at random with probability $\gamma^2/q_{ij}^\alpha$, a simple application of Chernoff's inequality yields that \ref{nr colours G_2} holds with probability at least $1-1/n^5$.
\endclaimproof

Hence, by Claim~\ref{claim G_2}, we may assume that $G_2$ is a spanning subgraph of $G_1$ such that properties~\ref{G_2 super-regular} and~\ref{nr colours G_2} hold.
By the construction of $G_2$, the restricted colouring $c|_{E(G_2)}$ is colour-split.

We show that also the required boundedness condition is satisfied, see~\eqref{eq:e^a G_2} below.
For $ij\in\binom{[r]}{2}$, we deduce from \eqref{eq:def p_ij} and~\ref{G_2 super-regular} that
\begin{align}\label{eq:e(G_2)}
e_{G_2}(V_i,V_j)= \gamma^2 p_{ij}(d\pm \heps^{1/2})n^2 \overset{\eqref{eq:def p_ij}}{=}
\big(d\pm \heps^{1/2}\big)\frac{\gamma^2n}{2\Delta}e_H(X_i,X_j).
\end{align}
For a colour $\alpha\in C$ and $ij\in\cbI$, as $G_2\subseteq G_1$, we obtain that
\begin{align}\label{eq:alpha edges in G''}
e^\alpha_{G_2}(V_i,V_j)
\leq e^\alpha_{G_1}(V_i,V_j)
\overset{\eqref{eq:I^alpha}}{\le}\frac{\gamma^2}{1-\binom{r}{2}\gamma^2}e^\alpha(G_1).
\end{align}
For a colour $\alpha\in C$ and $ij\in\cI^\alpha$, we obtain with \ref{nr colours G_2} that
\begin{align}\label{eq:alpha edges in G'' 2}
e^\alpha_{G_2}(V_i,V_j)
&\stackrel{\eqref{eq:p_ij^a}}{\leq} 
\frac{\gamma^2}{1-\vert\cbI\vert\gamma^2}\cdot
e^\alpha_{G_1}(V_i,V_j)
\frac{\sum_{i'j'\in\cI^\alpha}e^\alpha_{G_1}(V_{i'},V_{j'})}{e^\alpha_{G_1}(V_i,V_j)}+\eps n
\notag\\
&\stackrel{\hphantom{\eqref{eq:p_ij^a}}}{\leq} 
\frac{\gamma^2}{1-\binom{r}{2}\gamma^2}\sum_{i'j'\in\binom{[r]}{2}}e^\alpha_{G_1}(V_{i'},V_{j'})+\eps n
=\frac{\gamma^2}{1-\binom{r}{2}\gamma^2} e^\alpha(G_1)+\eps n.
\end{align}
Moreover, for every colour $\alpha\in C$ and every $ij\in\binom{[r]}{2}$, we conclude that
\begin{align*}
&
\frac{\gamma^2}{1-\binom{r}{2}\gamma^2}e^\alpha(G_1)+\eps n
\stackrel{\eqref{eq:e^alpha(G_1)}}{\leq} 
\frac{1-3\gamma/4}{1-\binom{r}{2}\gamma^2} \cdot \frac{\gamma^2n}{2\Delta}(d+\sqrt{\eps})+\eps n
\notag\\\stackrel{\eqref{eq:e(G_2)}}{\leq} &
\frac{1-3\gamma/4}{1-\binom{r}{2}\gamma^2}
\cdot\frac{e_{G_2}(V_i,V_j)}{e_{H}(X_i,X_j)}
\cdot\frac{d+\sqrt{\eps}}{d-\heps^{1/2}}+\eps n
\stackrel{\hphantom{\eqref{eq:e(G_2)}}}{\leq} 
\left(1-\frac{2\gamma}{3}\right)\frac{e_{G_2}(V_i,V_j)}{e_{H}(X_i,X_j)},
\end{align*}
which implies together with~\eqref{eq:alpha edges in G''} and~\eqref{eq:alpha edges in G'' 2} that for every colour $\alpha\in C$ and every $ij\in\binom{[r]}{2}$,
\begin{align}\label{eq:e^a G_2}
e^\alpha_{G_2}(V_i,V_j)\leq \left(1-\frac{2\gamma}{3}\right)\frac{e_{G_2}(V_i,V_j)}{e_{H}(X_i,X_j)}.
\end{align}

Let $G_3$ be a spanning subgraph of $G_2$ where for each bipartite pair $G_2[V_i,V_j]$ we keep each edge independently at random with probability $d'/(\gamma^2 p_{ij}d)$.
As $G_2[V_i,V_j]$ is $(\heps,\gamma^2p_{ij}d)$-super-regular, we may conclude by simple applications of Chernoff's inequality that with probability at least $1-1/n$ for all $ij\in\binom{[r]}{2}$, the graph $G_3[V_i,V_j]$ is $(\eps',d')$-super-regular, and for every colour $\alpha\in C$, we have $$e^\alpha_{G_3}(V_i,V_j)\leq \left(1-\frac{\gamma}{2}\right)\frac{e_{G_3}(V_i,V_j)}{e_{H}(X_i,X_j)}$$
due to~\eqref{eq:e^a G_2}.
Clearly, also $c$ restricted to $G_3$ is colour-split.
Hence, we conclude that there is a spanning subgraph $G_3$ of $G_2$ satisfying properties~\ref{G' blow-up instance}--\ref{c G' bounded}, which implies the statement with $G_3$ playing the role of~$G'$.
\end{proof}

The next lemma states that we can refine the partitions of a blow-up instance $(H,G, (X_i)_{i\in[r]}, (V_i)_{i\in[r]})$ where the edge-colouring of $G$ is colour-split such that $H$ only induces matchings between its refined partition classes and the bipartite pairs of $G$ are still super-regular and colour-split. 
Similar as in the reduction in~\cite{RR:99}, we first apply the Hajnal--Szemer\'edi theorem to $H^2[X_i]$ for each cluster $X_i$ to obtain a refined partition of $H$ where every cluster is now 2-independent. 
Accordingly, we refine the partition of $G$ randomly to preserve the super-regularity. 
Additionally, we partition the colours into disjoint colour sets such that the colouring between the refined partitions of $G$ is still colour-split. 

We first state the classical Hajnal--Szemer\'edi theorem.
\begin{theorem}[\cite{HS:70}] \label{thm:HS}
Let $G$ be a graph on $n$ vertices with $\Delta(G)< k \le n$. Then $V(G)$ can be partitioned into $k$ independent sets of size $\lfloor \frac{n}{k}\rfloor$ or $\lceil \frac{n}{k}\rceil$.
\end{theorem}

\begin{lemma}\label{lem:Hajnal Szemeredi partitioning}
Let $1/n\ll\eps\ll\eps'\ll d'\ll\gamma\ll d,1/\Lambda,1/r,1/\Delta$. 
Suppose $(H,G, (X_i)_{i\in[r]}, (V_i)_{i\in[r]})$ is an $(\eps,d)$-super-regular blow-up instance.
Assume further that 
\begin{enumerate}[label={\rm (\roman*)}]
	\item $\Delta(H)\leq \Delta$ and $e_H(X_i,X_j)\geq\gamma^2n$ for all $ij\in \binom{[r]}{2}$;
	\item $|V_i|=(1\pm\eps)n$ for all $i\in[r]$;
	\item $c\colon E(G)\to C$ is a colour-split edge-colouring such that $c$ is locally $\Lambda$-bounded and $c$ restricted to $G[V_i,V_j]$ is $(1-\gamma)e_G(V_i,V_j)/e_H(X_i,X_j)$-bounded for all $ij\in\binom{[r]}{2}$.
\end{enumerate}
Then there exists an $(\eps',d')$-super-regular blow-up instance  $(H',G',(X_{i,j})_{i\in [r],j\in[\Delta^2]},(V_{i,j})_{i\in [r],j\in[\Delta^2]})$ such that
\begin{enumerate}[label={\rm (\alph*)}]
\item\label{refined partitions} $(X_{i,j})_{j\in[\Delta^2]}$ is partition of $X_i$ and $(V_{i,j})_{j\in[\Delta^2]}$ is partition of $V_i$ for every $i\in[r]$, and $|X_{i,j}|=|V_{i,j}|=(1\pm\eps')n/\Delta^2$ for all $i\in[r], j\in[\Delta^2]$;
\item\label{H'} $H'$ is a supergraph of $H$ on $V(H)$ such that $H'[X_{i_1,j_1},X_{i_2,j_2}]$ is a matching of size at least $\gamma^4 n/\Delta^2$ for all $i_1,i_2\in[r], j_1,j_2\in[\Delta^2], (i_1,j_1)\neq(i_2,j_2)$;\COMMENT{Note that we do not need to require an upper bound for $\Delta(H')$ since $H'[X_{i_1,j_1},X_{i_2,j_2}]$ is a matching.}
\item\label{G'} $G'$ is a graph on $V(G)$ such that $G'[V_{i_1,j_1},V_{i_2,j_2}]\subseteq G[V_{i_1},V_{i_2}]$ for all distinct $i_1,i_2\in[r]$ and all $j_1,j_2\in[\Delta^2]$;
\item\label{c'} $c'\colon E(G')\to C'$ is an edge-colouring of $G'$ such that $c'|_{E(G)\cap E(G')}=c|_{E(G)\cap E(G')}$, and $c'$ is colour-split with respect to the partition $(V_{i,j})_{i\in [r],j\in[\Delta^2]}$, and $c'$ is locally $\Lambda$-bounded, and $c'$ restricted to $G'[V_{i_1,j_1},V_{i_2,j_2}]$ is $$\left(1-\frac{\gamma}{2}\right)\frac{e_{G'}(V_{i_1,j_1},V_{i_2,j_2})}{e_{H'}(X_{i_1,j_1},X_{i_2,j_2})}\text{-bounded}$$ for all $i_1,i_2\in[r], j_1,j_2\in[\Delta^2], (i_1,j_1)\neq(i_2,j_2)$.
\end{enumerate}
\end{lemma}
\begin{proof}
Since $c\colon E(G)\to C$ is colour-split, we may assume that $c$ is the union of edge-colourings $c_{i_1i_2}\colon E(G[V_{i_1},V_{i_2}])\to C_{i_1i_2}$ for $i_1i_2\in\binom{[r]}{2}$ where $C_{i_1i_2}\cap C_{i_1'i_2'}=\emptyset$ for distinct $i_1i_2,i_1'i_2'\in\binom{[r]}{2}$.

First, we apply Theorem~\ref{thm:HS} to $H^2[X_i]$ for every $i\in[r]$.
Since $\Delta(H^2[X_i])\leq \Delta^2-1$, there exists a partition of $X_i$ into 2-independent sets $X_{i,1},\ldots,X_{i,\Delta^2}$ in $H$ each of size $|X_i|/\Delta^2\pm 1=(1\pm2\eps)n'$, where $n':=n/\Delta^2$.
Hence for all $i_1,i_2\in[r], j_1,j_2\in[\Delta^2]$, the bipartite graph $H[X_{i_1,j_1},X_{i_2,j_2}]$ is a (possibly empty) matching.
Clearly, we can add a minimial number of edges to $H$ to obtain a supergraph $H'$ such that $H'[X_{i_1,j_1},X_{i_2,j_2}]$ is a matching of size at least $\gamma^4n'$ for all $i_1,i_2\in[r], j_1,j_2\in[\Delta^2], (i_1,j_1)\neq(i_2,j_2)$, which yields~\ref{H'}.\COMMENT{Note that we even add edges between previously empty pairs when $i_1=i_2$. 
When we apply Lemma~\ref{lem:Hajnal Szemeredi partitioning}, note that clearly, any rainbow embedding of $H'$ also yields a rainbow embedding of~$H$.} 

In order to  obtain~\ref{refined partitions}, we refine the partition of $V(G)$ accordingly. 
We claim that the following partitions exist.
For every $i\in[r]$, let $(V_{i,j})_{j\in[\Delta^2]}$ be a partition of $V_i$ such that $|V_{i,j}|=|X_{i,j}|$ for every $j\in[\Delta^2]$, and such that 
for all distinct $i_1,i_2\in[r]$, all $j_1,j_2\in[\Delta^2]$, and $v\in V_{i_1,j_1}\cup V_{i_2,j_2}$, we have
\begin{align}\label{eq:degree refined G}
\dg_{G[V_{i_1,j_1},V_{i_2,j_2}]}(v)=(d\pm3\eps)n'
\end{align}
and
\begin{align}\label{eq:colouring bounded}
c|_{E(G[V_{i_1,j_1},V_{i_2,j_2}])} \text{ is } (1-\gamma+\eps)\frac{(d+3\eps)n'^2}{e_H(X_{i_1},X_{i_2})}\text{-bounded.}
\end{align}

That such a partition exists can be seen by a probabilistic argument as follows: 
For each $i\in[r]$, let $\tau_i\colon V_i\to[\Delta^2]$ where $\tau_i(v)$ is chosen uniformly at random for every $v\in V_i$, all independently, and let $V_{i,j}:=\{v\in V_i\colon \tau_i(v)=j \}$ for every $j\in[\Delta^2]$.
McDiarmid's inequality together with a union bound implies that~\eqref{eq:degree refined G} and \eqref{eq:colouring bounded} hold with probability at least $1-\eul^{-\sqrt{n}}$.
Moreover, standard properties of the multinomial distribution yield that $|V_{i,j}|=|X_{i,j}|$ for all $i\in[r], j\in[\Delta^2]$ with probability at least $\Omega(n^{-\Delta^2r})$.\COMMENT{For the $k$-multinomial coefficient, we obtain with Stirling's formula that
$$\binom{n}{\frac{n}{k},\ldots,\frac{n}{k}}=\frac{n!}{\left(\frac{n}{k}\right)!^k}\sim \frac{\sqrt{2\pi n}\left(\frac{n}{e}\right)^n}{\left(\sqrt{\frac{2\pi n}{k}}\left(\frac{n}{ke}\right)^{n/k}\right)^k}=\frac{k^{k/2}k^n}{\left(2\pi n \right)^{(k-1)/2}}.$$
This implies that 
$$\prob{\forall j\in[\Delta^2]\colon\left|\tau_i(j)\right|=\frac{n}{\Delta^2}}=\Omega(n^{-\Delta^2}),$$
for every $i\in[r]$. Thus,
$$\prob{\forall i\in[r], j\in[\Delta^2]\colon\left|\tau_i(j)\right|=\frac{n}{\Delta^2}}=\Omega(n^{-r\Delta^2}).$$
}

Thus, for every $i\in[r]$, there exists a partition $(V_{i,j})_{j\in[\Delta^2]}$ of $V_i$ with the required properties.

Since $G[V_{i_1},V_{i_2}]$ is $(\eps,d)$-super-regular and due to~\eqref{eq:degree refined G}, it follows that for all distinct $i_1,i_2\in[r]$ and all $j_1,j_2\in[\Delta^2]$, the graph $G[V_{i_1,j_1},V_{i_2,j_2}]$ is $(2\Delta^2\eps,d)$-super-regular.\COMMENT{Since $G[V_{i_1},V_{i_2}]$ is $(\eps,d)$-super-regular, it follows that $G[V_{i_1,j_1},V_{i_2,j_2}]$ is $\frac{n}{(1+\eps)n/\Delta^2}\eps$-regular, thus say $2\Delta^2\eps$-regular.
Due to~\eqref{eq:degree refined G}, it follows that $G[V_{i_1,j_1},V_{i_2,j_2}]$ is $(2\Delta^2\eps,d)$-super-regular.}
By the construction of the supergraph $H'$, we have added at most $\gamma^4\Delta^2n$ edges to each pair $(X_{i_1},X_{i_2})$ in $H$.
Hence for all distinct $i_1,i_2\in[r]$,
\begin{align}\label{eq:e_H'}
&e_H(X_{i_1},X_{i_2})
\geq e_{H'}(X_{i_1},X_{i_2})-\gamma^4\Delta^2n
\geq e_{H'}(X_{i_1},X_{i_2})(1-\gamma^2\Delta^2),
\end{align}
where the last inequality holds since $e_{H'}(X_{i_1},X_{i_2})\geq e_{H}(X_{i_1},X_{i_2})\geq \gamma^2n$.
Now~\eqref{eq:colouring bounded} and~\eqref{eq:e_H'} imply that for all distinct $i_1,i_2\in[r]$ and all $j_1,j_2\in[\Delta^2]$, the colouring\COMMENT{$$(1-\gamma+\eps)\frac{(d\pm3\eps)n'^2}{e_H(X_{i_1},X_{i_2})}\leq \frac{(1-\gamma+\eps)}{(1-\Delta^2\gamma^2)}\frac{(d+3\eps)n'^2}{e_{H'}(X_{i_1},X_{i_2})}
\leq \left(1-\frac{3\gamma}{4}\right)\frac{(d+3\eps)n'^2}{e_{H'}(X_{i_1},X_{i_2})}.$$}
\begin{align}\label{eq:colouring bounded2}
c|_{E(G[V_{i_1,j_1},V_{i_2,j_2}])} \text{ is } \left(1-\frac{3\gamma}{4}\right)\frac{(d+3\eps)n'^2}{e_{H'}(X_{i_1},X_{i_2})}\text{-bounded.}
\end{align}

Next, we iteratively define spanning subgraphs $G_2\subseteq G_1\subseteq G$ and a supergraph $G'\supseteq G_2$ that satisfies the required properties in the statement.

First, we claim that there exists a spanning subgraph $G_1\subseteq G$ that is colour-split with respect to the partition $(V_{i,j})_{i\in [r],j\in[\Delta^2]}$ and still super-regular.
In order to see that such a subgraph exists, we use a probabilistic argument.
For all distinct $i_1,i_2\in[r]$, let $\tau_{i_1i_2}\colon C_{i_1i_2}\to[\Delta^2]\times[\Delta^2]$ where each $\tau_{i_1i_2}(\alpha)$ is chosen independently at random according to the probability distribution $(p_{(i_1,j_1),(i_2,j_2)})_{j_1,j_2\in[\Delta^2]}$ with\COMMENT{For fixed but distinct $i_1,i_2\in[r]$, we have $\sum_{j_1,j_2\in[\Delta^2]}p_{(i_1,j_1),(i_2,j_2)}=1$.}
\begin{align}\label{eq:def p_i1j1i2j2}
p_{(i_1,j_1),(i_2,j_2)}:=\frac{e_{H'}(X_{i_1,j_1},X_{i_2,j_2})}{e_{H'}(X_{i_1},X_{i_2})}\geq \frac{\gamma^4n'}{2\Delta^4 n'}\geq\gamma^5.
\end{align}
Define $G_1$ by keeping each edge  $e\in E(G[V_{i_1,j_1},V_{i_2,j_2}])$ if $\tau_{i_1i_2}(c(e))=(j_1,j_2)$.
By Lemma~\ref{lem:regularity colour splitting} and since $G[V_{i_1,j_1},V_{i_2,j_2}]$ is $(2\Delta^2\eps,d)$-super-regular, there exists $G_1\subseteq G$ such that the colouring of $G_1$ is colour-split and $G_1[V_{i_1,j_1},V_{i_2,j_2}]$ is $(\eps'/2,p_{(i_1,j_1),(i_2,j_2)}d)$-super-regular for all distinct $i_1,i_2\in[r]$ and all $j_1,j_2\in[\Delta^2]$.

For all distinct $i_1,i_2\in[r]$, all $j_1,j_2\in[\Delta^2]$, and every colour $\alpha\in C_{i_1i_2}$, we obtain
\begin{align*}
e^\alpha_{G_1}(V_{i_1,j_1},V_{i_2,j_2})
&\leq e^\alpha_{G}(V_{i_1,j_1},V_{i_2,j_2})
\stackrel{\eqref{eq:colouring bounded2}}{\leq} \left(1-\frac{3\gamma}{4}\right)\frac{(d+3\eps)n'^2}{e_{H'}(X_{i_1},X_{i_2})}
\notag\\
&=\left(1-\frac{3\gamma}{4}\right)\frac{p_{(i_1,j_1),(i_2,j_2)}(d+3\eps)n'^2}{e_{H'}(X_{i_1,j_1},X_{i_2,j_2})},
\end{align*}
and thus, since $G_1[V_{i_1,j_1},V_{i_2,j_2}]$ is $(\eps'/2,p_{(i_1,j_1),(i_2,j_2)}d)$-super-regular, we conclude that
\begin{align}\label{eq:HS e^a G_1}
e^\alpha_{G_1}(V_{i_1,j_1},V_{i_2,j_2})
\leq\left(1-\frac{2\gamma}{3}\right)\frac{e_{G_1}(V_{i_1,j_1},V_{i_2,j_2})}{e_{H'}(X_{i_1,j_1},X_{i_2,j_2})}.
\end{align}

Let $G_2$ be the spanning subgraph of $G_1$ where for each bipartite pair $G_1[V_{i_1,j_1},V_{i_2,j_2}]$, we keep each edge independently at random with probability $d'/(p_{(i_1,j_1),(i_2,j_2)}d)$. 
As $G_1[V_{i_1,j_1},V_{i_2,j_2}]$ is $(\eps'/2,p_{(i_1,j_1),(i_2,j_2)}d)$-super-regular, we may conclude by simple applications of Chernoff's inequality\COMMENT{We can apply Chernoff's inequality to the set $\{Z_e\colon e\in E(G_1[V_{i_1,j_1},V_{i_2,j_2}])\}$ of independent Bernoulli random variables where $Z_e$ indicates whether $e\in E(G_1[V_{i_1,j_1},V_{i_2,j_2}])$ is present in $G_2[V_{i_1,j_1},V_{i_2,j_2}]$. 
It holds $$\dg_{G_2[V_{i_1,j_1},V_{i_2,j_2}]}(v)=\sum_{v\in e} Z_e$$
and 
$$e_{G_2}(S,T)=\sum_{st\in E(G_1),s\in S,t\in T} Z_{st},$$
for $S\subseteq V_{i_1,j_1}$, $T\subseteq V_{i_2,j_2}$. 
Similarly, 
$$e_{G_2}^\alpha(V_{i_1,j_1},V_{i_2,j_2})=\sum_{e\in E(G_1[V_{i_1,j_1},V_{i_2,j_2}]), c(e)=\alpha} Z_e.$$
Hence, Chernoff's inequality yields the required concentration since the random variables $Z_e$ are independent.} that with probability at least $1-1/n$ for all distinct  $i_1,i_2\in[r]$ and all $j_1,j_2\in[\Delta^2]$, the graph $G_2[V_{i_1,j_1},V_{i_2,j_2}]$ is $(\eps',d')$-super-regular, and by~\eqref{eq:HS e^a G_1} for every colour $\alpha\in C$, we have 
\begin{align}\label{eq:HS e^a G_2}
e^\alpha_{G_2}(V_{i_1,j_1},V_{i_2,j_2})\leq \left(1-\frac{\gamma}{2}\right)\frac{e_{G_2}(V_{i_1,j_1},V_{i_2,j_2})}{e_{H'}(X_{i_1,j_1},X_{i_2,j_2})}.
\end{align}

Finally, we may add edges in the empty bipartite graphs $G_2[V_{i,j},V_{i,j'}]$ for all $i\in[r]$ and all distinct $j,j'\in[\Delta^2]$ in such a way that we obtain a supergraph $G'\supseteq G_2$ where $G'[V_{i_1,j_1},V_{i_2,j_2}]$ is $(\eps',d')$-super-regular for all $i_1,i_2\in[r]$ and $j_1,j_2\in[\Delta^2]$, $(i_1,j_1)\neq(i_2,j_2)$.
Hence, we conclude that $(H',G',(X_{i,j})_{i\in [r],j\in[\Delta^2]},(V_{i,j})_{i\in [r],j\in[\Delta^2]})$ is an $(\eps',d')$-super-regular blow-up instance that satisfies~\ref{G'}.

Let $c^{art}\colon\binom{V(G)}{2}\to C^{art}$ be a rainbow edge-colouring of all possible edges $\binom{V(G)}{2}$ such that $C^{art}\cap C=\emptyset$. 
By colouring the edges $E(G')\sm E(G_2)$ using $c^{art}$, we may obtain an edge-colouring $c'\colon E(G')\to C\cup C^{art}$ which extends $c$ and is clearly $\Lambda$-bounded. 
By the construction of $G_2$, the colouring $c'$ is colour-split, and
$$\left(1-\frac{\gamma}{2}\right)\frac{e_{G'}(V_{i_1,j_1},V_{i_2,j_2})}{e_{H'}(X_{i_1,j_1},X_{i_2,j_2})}\text{-bounded}$$
for each bipartite subgraph $G'[V_{i_1,j_1},V_{i_2,j_2}]$ with $i_1,i_2\in[r], j_1,j_2\in[\Delta^2], (i_1,j_1)\neq(i_2,j_2)$ due to~\eqref{eq:HS e^a G_2}.
This yields~\ref{c'} and completes the proof.
\end{proof}

%%%%%%%%%%%%%%%%%%%%%%%%%%%%%%%%%%%%%%%%%%%%%%%%%%%%%%%%%%%%%%%%%%%%%%%%%%%%%%%

\section{Approximate Embedding Lemma}\label{sec:4graphs new2}
In this section, we prove the `Approximate Embedding Lemma' (Lemma~\ref{lem:embedding lemma}), which allows us to embed a cluster $X_i$ into $V_i$ (here $X_0,V_0$) almost completely, while maintaining crucial properties of the `candidacy graphs' of other clusters.

We say that $(H,G, (A_i)_{i\in[r]_0},c)$ is an \defn{embedding-instance} if 
\begin{itemize}
\item $H,G$ are graphs and $A_i$ is a bipartite graph with vertex partition $(X_i,V_i)$ for every $i\in[r]_0$ such that $(X_i)_{i\in[r]_0}$ is a partition of $V(H)$ into independent sets, $(V_i)_{i\in[r]_0}$ is a partition of $V(G)$, and $|X_i|=|V_i|$ for all $i\in[r]_0$;
\item for all $i\in[r]$, the graph $H[X_0,X_i]$ is a matching;
\item $c\colon E(G\cup \bigcup_{i\in[r]_0}A_i) \to 2^C$ is an edge set colouring that is colour-split with respect to the partition $(X_0,\ldots,X_r,V_0,\ldots,V_r)$
and satisfies $|c(e)|=1$ for all $e\in E(G)$.
\end{itemize}

We say that $(H,G, (A_i)_{i\in[r]_0},c)$ is an \emph{$(\eps,(d_i^G)_{i\in[r]},(d_i)_{i\in[r]_0},t,\Lambda)$-embedding-instance} if in addition, we have that
\begin{itemize}
\item $G[V_0,V_i]$ is $(\eps,d_i^G)$-super-regular and $c$ restricted to $G[V_0,V_i]$ is $(1+\eps)e_G(V_0,V_i)/e_H(X_0,X_i)$-bounded for all $i\in[r]$;
\item $A_i$ is $(\eps,d_i)$-super-regular and $c$ restricted to $A_i$ is $(1+\eps)d_i |X_i|$-bounded for all $i\in[r]_0$;
\item $c$ is locally $\Lambda$-bounded and $|c(e)|\le t$ for all $e\in  \bigcup_{i\in[r]_0}E(A_i)$.
\end{itemize}

Here, $X_0$ is the cluster we want to embed into $V_0$ by finding an almost perfect rainbow matching $\sigma$ in~$A_0$, and $t$ can be thought of as the number of clusters we have previously embedded. 
For convenience, we identify matchings $\sigma$ between $X_0$ and $V_0$ with functions $\sigma\colon X_0^\sigma\to  V_0^\sigma$, where $X_0^\sigma=V(\sigma)\cap X_0$ and $ V_0^\sigma=V(\sigma)\cap V_0$.
Whenever we write $xv\in E(A_i)$, we tacitly assume that $x\in X_i$ and $v\in V_i$.

The following two definitions encapsulate how the choice of $\sigma$ affects the candidacy graphs $(A_i)_{i\in[r]}$ and their colouring for the next step (see Figure~\ref{fig:induction step}). Let $(H,G,(A_i)_{i\in[r]_0},c)$ be an embedding-instance.

\begin{defin}[Updated candidacy graphs]\label{def:updated candidacy}
For a matching $\sigma\colon X_0^\sigma\to  V_0^\sigma$ in~$A_0$, 
we define $(\Aupd_i)_{i\in[r]}$ as the \emph{updated candidacy graphs (with respect to $\sigma$)} as follows: for every $i\in[r]$, let $\Aupd_i$ be the spanning subgraph of $A_i$ containing precisely those edges $xv\in E(A_i)$ for which the following holds: 
if $x$ has an $H$-neighbour $x_0\in X_0^\sigma$ (which would be unique), then $\sigma(x_0)v\in E(G[V_0,V_i])$.
\end{defin}

This definition ensures that when we embed $x$ in a future round, we are guaranteed that the $H$-edge $x_0x$ is mapped to a $G$-edge. Note that this definition does not depend at all on the colouring~$c$.
Moreover, we also define updated colourings for the updated candidacy graphs, where we add up to one additional colour to the edges in the new candidacy graphs according to~$\sigma$.

\begin{defin}[Updated colouring]\label{def:updated colouring}
For a matching $\sigma\colon X_0^\sigma\to  V_0^\sigma$ in~$A_0$, we define the \defn{updated edge set colouring $c^\sigma$} of the updated candidacy graphs as follows: for each $i\in[r]$ and $xv\in E(\Aupd_i)$, when $x$ has an $H$-neighbour $x_0\in X_0^\sigma$, then set $c^\sigma(xv) := c(xv) \cup c(\sigma(x_0)v)$, and otherwise set $c^\sigma(xv):=c(xv)$.\COMMENT{Note that $\sigma(x_0)v\in E(G)$ because $xv\in E(\Aupd_i)$.}
\end{defin}

\begin{figure}[t]%
	\begin{center}
		\begin{tikzpicture}[scale = 0.7,every text node part/.style={align=center}]
		\def\ver{0.1} %size of a vertex

		\draw (14,5.5) node {$H$};
		\draw (14,-0.5) node {$G$};
		
		\draw (1,6.5) node {$X_{0}$};
		\draw (7,6.5) node {$X_{i}$};
		\draw (13,6.5) node {$X_{j}$};

		\draw (1,-1.5) node {$V_{0}$};
		\draw (7,-1.5) node {$V_{i}$};
		\draw (13,-1.5) node {$V_{j}$};
		
		\draw (7,2.5) node {$A_{i}$};

		\filldraw[fill=black!20, draw=black!50]
		(-0.7,6)--(0.7,6) -- (0.7,0) -- (-0.7,0) -- cycle;

		\draw[thick,->, dashed] (9.4,-1.4)  arc  (180:118:4.4cm);

		\draw[very thick, fill=white] 
		(0,0) ellipse (1 and 1.5)
		(6,0) ellipse (1 and 1.5)
		(12,0) ellipse (1 and 1.5)
		(0,5) ellipse (1 and 1.5)
		(6,5) ellipse (1 and 1.5)
		(12,5) ellipse (1 and 1.5);

		\draw[thick]
		(0,5) .. controls (2.5,8) and (9.5,8) .. (12,5)
		(0,5)--(6,5)
		(0,0) -- node[midway,anchor=north] {$\alpha$} (6,0)
		(0,0) .. controls (2.5,-3) and (9.5,-3) .. node[near end,anchor=north] {$\alpha$} (12,0)
		
		(12,0) -- node[anchor=west] {$\{\alpha\}\cup c(x_jv_j)$} (12,5)  
		;

		\draw[line width=3pt,->]
		(0,3.5) -- node[anchor=west] {$\sigma \In E(A_0)$}(0,1.5);

		\draw[fill=black!20, thick]
		(6,5) -- (5.4,0.4) -- (6.6,0.4) -- (6,5);
		
		\draw[fill=black!20,thick]
		(0,0) -- (6,0.3) -- (6,-1.2) -- (0,0);
		
		\draw[fill=black!50, thick] 
		(6,-0.4) ellipse (0.6 and 0.8)
		(6,0.4) ellipse (0.6 and 0.8);

		\draw[very thick] 
		(6,0) ellipse (1 and 1.5)
		;

		\begin{scope}[even odd rule]
		\draw[fill=black!20, thick] 
		(6,-0.4) ellipse (0.6 and 0.8)
		(6,0.4) ellipse (0.6 and 0.8);
		\end{scope}

		\draw[fill]
		(12,5) circle (\ver)
		node[anchor=south west] {$x_j$}
		(6,5) circle (\ver)
		node[anchor=south east] {$x_i$}
		(0,5) circle (\ver)
		node[anchor=south east] {$x_0$}
		(6,-.1) circle (\ver)
		(0,0) circle (\ver)
		node[anchor=north east] {$v_0$}
		(12,0) circle (\ver)
		node[anchor=north west] {$v_j$}
		;
		
		\draw[thick]
		(0,0)--(6,-.1)--(6,5);
		
		\draw (6.25,-0.7) node {$v_{i}$};

		\end{tikzpicture}
	\end{center}
	\caption{If $x_0$ is mapped to $v_0$ by~$\sigma$, then only those candidates of $x_i$ remain that are neighbours of~$v_0$. Moreover, colour $\alpha$ of the edge $v_0v_j$ is added to the the candidate edge $x_jv_j$, which captures the information that if $x_j$ is later embedded at $v_j$, then this embedding uses~$\alpha$.}
	\label{fig:induction step}
\end{figure}
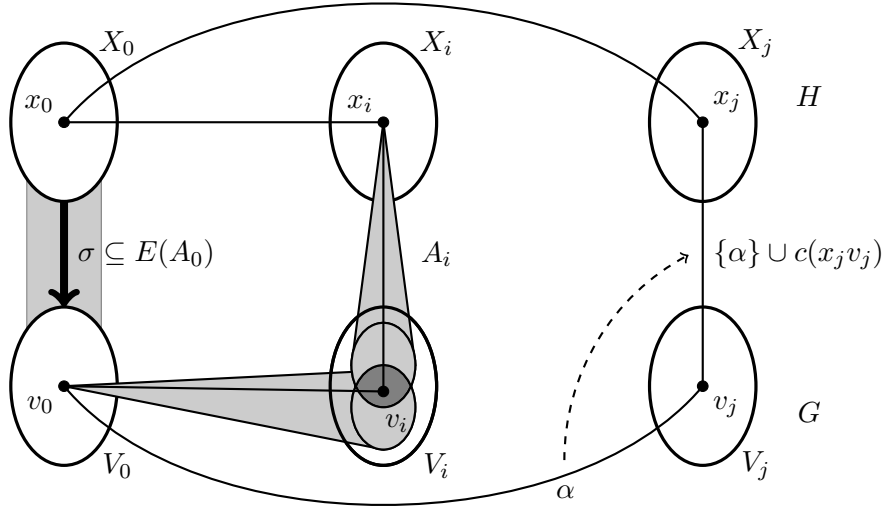

We now state and prove our Approximate Embedding Lemma.
\begin{lemma}[Approximate Embedding Lemma]\label{lem:embedding lemma}
Let $1/n\ll\eps\ll\eps'\ll(d_i^G)_{i\in[r]},(d_i)_{i\in[r]_0},1/\Lambda,1/r,1/(t+1)$.
Suppose $(H,G,(A_i)_{i\in[r]_0},c)$ is an $(\eps,(d_i^G)_{i\in[r]},(d_i)_{i\in[r]_0},t,\Lambda)$-embedding-instance with $|V_0|=n$, $|V_i|=(1\pm\eps)n$, and $e_H(X_0,X_i)\geq \eps'n$ for all $i\in[r]$.
Suppose the codegree of $c$ is $K\leq\sqrt{n}$.

Then there is a rainbow matching $\sigma\colon X_0^\sigma\to  V_0^\sigma$ in $A_0$ of size at least $(1-\eps')n$ such that for all $i\in[r]$, there exists a spanning subgraph $A_i^{new}$ of the updated candidacy graph $\Aupd_i$ and
\begin{enumerate}[label={\rm (\Roman*)\textsubscript{\ref{lem:embedding lemma}}}]
\item\label{cond:4graphs i} $A_i^{new}$ is $(\eps',d_i^Gd_i)$-super-regular;
\item\label{cond:4graphs ii} the updated colouring $c^\sigma$ restricted to $A_i^{new}$ is $(1+\eps') d_i^Gd_i |X_i|$-bounded;
\item\label{cond:4graphs iii} $c^\sigma$ restricted to $A_i^{new}$ has codegree at most $\max\{K,n^\eps\}$.
\end{enumerate}
\end{lemma}

We split the proof into three steps. 
In Step~\ref{step:codegree} we remove non-typical vertices and edges in order to guarantee that certain neighbourhoods intersect appropriately.
In Step~\ref{step:hypergraph} we use a suitable hypergraph construction together with Theorem~\ref{thm:hypermatching} to obtain the required rainbow matching~$\sigma$. 
By defining certain weight functions in Step~\ref{step:weight functions}, we utilise the conclusions of Theorem~\ref{thm:hypermatching} to show that $\sigma$ can be chosen such that~\ref{cond:4graphs i}--\ref{cond:4graphs iii} hold.

\begin{proof}
Without loss of generality we may assume that $|c(e)|=t$ for all $e\in E(A_0)$. (Otherwise, we may simply add new `dummy' colours in such a way that the obtained colouring still satisfies the conditions of the lemma, and these colours can simply be deleted afterwards.)\COMMENT{We only need that $|c(e)|=t$ in order to obtain a $(t+2)$-uniform hypergraph in Step~\ref{step:hypergraph}.}

We also choose a new constant $\heps$ such that $\eps\ll\heps\ll\eps'$.

\begin{step}\label{step:codegree}
Removing non-typical vertices and edges
\end{step}

Let $ \Haux$ be an auxiliary supergraph of $H$ that is obtained by adding a maximal number of edges between $X_0$ and $X_i$ for every $i\in[r]$ subject to $\Haux[X_0,X_i]$ being a matching
(note that $e_{\Haux}(X_0,X_i)\geq (1-\eps)n$).
Next, we define subgraphs of $G$ and $(A_i)_{i\in[r]_0}$ to achieve that certain neighbourhoods intersect appropriately.

Let $\Abad_0$ be the spanning subgraph of $A_0$ such that an edge $x_0v_0\in E(A_0)$ belongs to $\Abad_0$ if there is some $i\in[r]$ with $\{x_i\}=N_{ \Haux}(x_0)\cap X_i$ and
\begin{align}\label{eq:codegree F_12'}
|N_{A_i}(x_i)\cap N_{G}(v_0)|\neq (d_i^Gd_i\pm 3\eps)|V_i|.
\end{align}
For $i\in[r]$, let $\Abad_i$ be the spanning subgraph of $A_i$ such that an edge $x_iv_i\in E(A_i)$ belongs to $\Abad_i$ if $\{x_0\}=N_{ \Haux}(x_i)\cap X_0$ and
\begin{align}\label{eq:codegree F_13'}
|N_{A_0}(x_0)\cap N_{G}(v_i)|\neq (d_i^Gd_0\pm 3\eps )|V_0|.
\end{align}
Let $\Gbad$ be the spanning subgraph of $G$ such that an edge $v_0v_i\in E(G[V_0,V_i])$ belongs to $\Gbad[V_0,V_i]$ for $i\in[r]$ whenever 
\begin{align}\label{eq:codegree F_23'}
e_H(N_{A_0}(v_0), N_{A_i}(v_i))\neq (d_0d_i \pm3\eps)e_H(X_0,X_i).
\end{align}

Using Fact~\ref{fact:regularity}, it is easy to see that $\Delta(\Abad_0)\le 3r\eps n$ and $\Delta(\Abad_i)\le 3\eps |V_i|$ for each $i\in[r]$. 
We also claim that for each $i\in[r]_0$, there exists $V_i^{bad}\In V_i$ with $|V_i^{bad}|\le 3r\eps n$, such that all vertices not in $V_0^{bad}\cup \dots \cup V_r^{bad}$ have degree at most $3r\eps n$ in~$\Gbad$. Indeed, fix $i\in[r]$ and let $\tX_0:=N_H(X_i)$ and $\tX_i:=N_H(X_0)$. Recall that $|\tX_0|=|\tX_i|=e_H(X_0,X_i)\ge \eps' n$. Using Fact~\ref{fact:regularity}, there exists $V_i^{bad}\In V_i$ with $|V_i^{bad}|\le 3\eps |V_i|$ such that all $v_i\in V_i\sm V_i^{bad}$ satisfy $|N_{A_i}(v_i)\cap \tX_i|=(d_i\pm \eps)|\tX_i|$.
Now, fix such a vertex~$v_i$. 
Let $U:=N_H(N_{A_i}(v_i))$. 
Using Fact~\ref{fact:regularity} again, we can see that all but at most $3\eps n$ vertices $v_0\in V_0$ satisfy $|N_{A_0}(v_0)\cap U|=(d_0\pm \eps)|U|=(d_0\pm \eps)(d_i\pm \eps)e_H(X_0,X_i)$. 
  Hence, $\dg_{\Gbad}(v_i)\le 3\eps n$. 
Similarly, one can see that there exists $V_{0,i}^{bad}\In V_0$ with $|V_{0,i}^{bad}|\le 3\eps n$ such that all $v_0\in V_0\sm V_{0,i}^{bad}$ satisfy $|N_{\Gbad}(v_0)\cap V_i|\le 3\eps n$. 
Let $V_0^{bad}:=\bigcup_{i=1}^r V_{0,i}^{bad}$. Then $V_0^{bad}, \dots, V_r^{bad}$ are as desired.

Now, let
\begin{align*}
	\Agood_0 &:=A_0[X_0,V_0\sm V_0^{bad}]-E(\Abad_0), \quad  \Agood_i:=A_i-E(\Abad_i),\\
	\Ggood_{0i} &:=G[V_0\sm V_0^{bad}, V_i]-E(G^{bad}[V_0, V_i\sm V_i^{bad}]),	 \quad\mbox{ for all $i\in[r]$}.
\end{align*}
Since we only seek an almost perfect rainbow matching $\sigma$ in $A_0$, we can remove the vertices $V_0^{bad}$ from $A_0$ and find $\sigma$ in $\Agood_0$. 
By keeping the vertices $V_i^{bad}$ for $i\in[r]$ and the corresponding edges $E(G[V_0,V_i^{bad}])$ in $\Ggood_{0i}$, we can guarantee that the candidacy graphs $\Agood_i$ are still spanning subgraphs of~$A_i$.

By Fact~\ref{fact:regularity robust}, we have that $\Ggood_{0i}$ is $(\heps,d_i^G)$-super-regular, that $\Agood_0$ is $(\heps,d_0)$-super-regular and that $\Agood_i$ is $(\heps,d_i)$-super-regular.
Crucially, we now have the following properties.
\begin{align}
\begin{split}\label{eq:codegree F_12}
&|N_{\Agood_i}(x_i)\cap N_{\Ggood_{0i}}(v_0)|=(d_i^Gd_i\pm\heps)|V_i|,
\\& \text{for all  $x_0v_0\in E(\Agood_0)$ whenever $\{x_i\}=N_{ \Haux}(x_0)\cap X_i$, $i\in[r]$;} 
\end{split}
\\[1em]
\begin{split}\label{eq:codegree F_13}
&|N_{\Agood_0}(x_0)\cap N_{\Ggood_{0i}}(v_i)|=(d_i^Gd_0\pm\heps)|V_0|, 
\\&\text{for all  $x_iv_i\in E(\Agood_i)$, $i\in[r]$,  whenever $\{x_0\}=N_{ \Haux}(x_i)\cap X_0$;}
\end{split}
\\[1em]
\begin{split}\label{eq:codegree F_23}
&e_H(N_{\Agood_0}(v_0), N_{\Agood_i}(v_i))=(d_0d_i\pm\heps)e_H(X_0,X_i),
\\&\text{for all  $v_0v_i\in E(\Ggood_{0i}-V_i^{bad})$ and $i\in[r]$.}
\end{split}
\end{align}

Indeed, consider $x_0v_0\in E(\Agood_0)$ with $\{x_i\}=N_{ \Haux}(x_0)\cap X_i$. By~\eqref{eq:codegree F_12'}, we have $|N_{A_i}(x_i)\cap N_{G}(v_0)|= (d_i^Gd_i\pm 3\eps)|V_i|$. Moreover, $v_0\notin V_0^{bad}$. 
Hence, $\dg_{\Abad_i}(x_i),\dg_{\Gbad}(v_0)\le 3r\eps n$, which implies~\eqref{eq:codegree F_12}.
Similar arguments hold for~\eqref{eq:codegree F_13} and~\eqref{eq:codegree F_23}. \COMMENT{For \eqref{eq:codegree F_13}, consider $x_iv_i\in E(\Agood_i)$ with $\{x_0\}=N_{ \Haux}(x_i)\cap X_0$. By \eqref{eq:codegree F_13'}, we have $|N_{A_0}(x_0)\cap N_{G}(v_i)|= (d_i^Gd_0\pm 3\eps )|V_0|$. If $v_i\in V_i^{bad}$, then all edges are added back, if $v_i\notin V_i^{bad}$, then degree is small. Plus edges to $V_0^{bad}$ might be lost.
For \eqref{eq:codegree F_23}, consider $v_0v_i\in E(\Ggood_{0i}-V_i^{bad})$. This means $v_0v_i\notin G^{bad}$. Hence, by \eqref{eq:codegree F_23'}, we have $e_H(N_{A_0}(v_0), N_{A_i}(v_i))= (d_0d_i \pm3\eps)e_H(X_0,X_i)$. Since $v_0\notin V_0^{bad}$, $v_0\in V(\Agood_0)$. Bad degrees are small.}

\begin{step}\label{step:hypergraph}
Constructing an auxiliary hypergraph
\end{step}

We aim to apply Theorem~\ref{thm:hypermatching} to find the required rainbow matching $\sigma$.
To this end, let $f_e:=e\cup c(e)$ for $e\in E(\Agood_0)$ and let $\cH$ be the $(t+2)$-uniform hypergraph $\cH$ with vertex set $X_0\cup V_0\cup C$\COMMENT{Could also define $\cH$ on $X_0\cup (V_0\sm V_0^{bad})\cup C$. However, note that the hyperedges are only defined for $e\in E(\Agood_0)$.}
and edge set $\{f_e\colon e\in E(\Agood_0) \}$.
A key property of the construction of $\mathcal{H}$ is a bijection between rainbow matchings $M$ in $\Agood_0$ and matchings $\cM$ in $\cH$ by assigning $M$ to $\cM=\{f_e\colon e\in M \}$.

In order to apply Theorem~\ref{thm:hypermatching}, we first establish upper bounds on $\Delta(\cH)$ and $\Delta^c(\cH)$. 
Since $\Agood_0$ is $(\heps,d_0)$-super-regular, $|X_0|=n$, and $c$ restricted to $\Agood_0$ is $(1+\eps)d_0n$-bounded, we conclude that
\begin{align}
\Delta(\mathcal{H})\leq (d_0+\heps)n.
\label{eq:Delta H}
\end{align}

Let $\Delta:=(d_0+\heps)n$.
Since $c$ is locally $\Lambda$-bounded, the codegree in $\cH$ of a vertex in $X_0\cup V_0$ and a colour in $C$ is at most $\Lambda$.
By assumption, the codegree in $\cH$ of two colours in $C$ is at most $K$.
For two vertices in $X_0\cup V_0$, the codegree in $\cH$ is at most~$1$.
Altogether, this implies that
\begin{align}
\Delta^c(\mathcal{H})
\leq \sqrt{n}
\leq \Delta^{1-\eps^2}.
\label{eq:Delta^c H}
\end{align}

Suppose $\cW$ is a set of given weight functions $\omega\colon E(\Agood_0)\to[\Lambda]_0$ with $|\cW|\leq n^5$ (which we will explicitly specify in Step~\ref{step:weight functions} to establish~\ref{cond:4graphs i}--\ref{cond:4graphs iii}.)
Note that every weight function $\omega\colon E(\Agood_0)\to[\Lambda]_0$ naturally corresponds to a  weight function $\omega_\cH\colon E(\cH)\to [\Lambda]_0$ by defining $\omega_\cH(f_e):= \omega(e)$.
If $\omega(E(A_0'))\geq n^{1+\eps/2}$, define $\tilde{\omega}:=\omega$. 
Otherwise, arbitrarily choose $\tilde{\omega}\colon E(\Agood_0)\to[\Lambda]_0$ such that $\omega\le \tilde{\omega}$ and $\tilde{\omega}(E(A_0'))= n^{1+\eps/2}$. 
By~\eqref{eq:Delta H} and~\eqref{eq:Delta^c H}, we can apply Theorem~\ref{thm:hypermatching} (with $(d_0+\heps)n,\eps^2,t+2,\set{\tilde{\omega}_\cH}{\omega \in \cW}$ playing the roles of $\Delta,\delta,r,\cW$) to obtain a matching $\cM$ in $\cH$ that corresponds to a rainbow matching $M$ in $\Agood_0$ that satisfies the following property by the conclusion of Theorem~\ref{thm:hypermatching}:
\begin{align}
\omega(M)&
=(1\pm \heps^{1/2})\frac{\omega(E(\Agood_0))}{d_0 n},\text{ for all $\omega\in\cW$ with $\omega(E(A_0'))\geq n^{1+\eps/2}$;} \label{eq:w(M) big}\\
\omega(M)&\le \max\Set{(1+\heps^{1/2})\frac{\omega(E(\Agood_0))}{d_0 n},n^\eps} \text{ for all $\omega\in\cW$.} \label{eq:w(M) upper bound}
\end{align}
\COMMENT{Note that $\omega(E(A_0'))\geq n^{1+\eps/2}\geq \Lambda n^{1+\eps^2}\geq \max_{e\in E(\cH)}\omega(e)\Delta^{1+\eps^2}$, and thus the conclusion of Theorem~\ref{thm:hypermatching} implies~\eqref{eq:w(M) big}.
\\For~\eqref{eq:w(M) upper bound} and $\omega\in\cW$ with $\omega(E(A_0'))<n^{1+\eps/2}$, note that $\omega(M)\leq\tilde{\omega}(M)\leq (1+\eps)\frac{n^{1+\eps/2}}{d_0n}\leq n^\eps$.}Let $\sigma\colon X_0^\sigma\to  V_0^\sigma$ be the function given by the matching $M$, where $X_0^\sigma=X_0\cap V(M)$ and $ V_0^\sigma=V_0\cap V(M)$.

One way to exploit~\eqref{eq:w(M) big} is to control the number of edges in $M$ between sufficiently large sets of vertices. 
To this end, for subsets $S\subseteq X_0$ and $T\subseteq V_0$ such that $|S|,|T|\geq2\heps n$, we define a weight function $\omega_{S,T}\colon E(\Agood_0)\to[\Lambda]_0$ with
\begin{align}\label{eq:def w_ST}
\omega_{S,T}(e):=
\begin{cases}
1 & \mbox{if } e\in E(\Agood_0[S,T\sm V_0^{bad}]),  \\
0 & \mbox{otherwise.}
\end{cases}
\end{align}
That is, $\omega_{S,T}(M)$ counts the number of edges between $S$ and $T$ that lie in~$M$. 
Since $A_0'$ is $(\heps,d_0)$-super-regular, \eqref{eq:w(M) big} implies (whenever $\omega_{S,T}\in \cW$) that
\begin{align}\label{eq:w_ST}
|\sigma\big(S\cap X_0^\sigma\big)\cap T|
=
\omega_{S,T}(M)
\stackrel{\eqref{eq:w(M) big}}{=}(1\pm \heps^{1/2})\frac{e(\Agood_0[S,T\sm V_0^{bad}])}{d_0 n}
=(1\pm 2\heps^{1/2})\frac{|S||T|}{n}.
\end{align}

\begin{step}\label{step:weight functions}
Employing weight functions to conclude \ref{cond:4graphs i}--\ref{cond:4graphs iii}
\end{step}

By Step \ref{step:hypergraph}, we may assume that~\eqref{eq:w(M) big} holds for a set of weight functions~$\cW$ that we will define during this step.
We will show that for this choice of $\cW$ the matching $\sigma\colon X_0^\sigma\to  V_0^\sigma$ as obtained in Step~\ref{step:codegree} satisfies \ref{cond:4graphs i}--\ref{cond:4graphs iii}.
Similar as in Definition~\ref{def:updated candidacy} (here with $H$ replaced by $H^+$), we define subgraphs $(\Agoodupd_i)_{i\in[r]}$ of $(\Agood_i)_{i\in[r]}$ as follows.
For every $i\in[r]$, let $\Agoodupd_i$ be the spanning subgraph of $\Agood_i$ containing precisely those edges $xv\in E(\Agood_i)$
for which the following holds: if $\{x_0\}=N_{ \Haux}(x)\cap X_0^\sigma$, then $\sigma(x_0)v\in E(\Ggood_{0i})$.
Since $\Agood_i\subseteq A_i$ and due to the construction of $\Agoodupd_i$, we conclude that $\Agoodupd_i$ is a spanning subgraph of the updated candidacy graph $\Aupd_i$ (with respect to $\sigma$) for every $i\in[r]$ (see~Definition~\ref{def:updated candidacy}).
By taking a suitable subgraph of $\Agoodupd_i$ we will later obtain the required candidacy graph~$A_i^{new}$.

\medskip
First, we show that the matching $M$ has size at least $(1-2\heps^{1/2})n$.
Adding $\omega_{X_0,V_0}$ as defined in~\eqref{eq:def w_ST} to $\cW$ and using~\eqref{eq:w_ST} yields
\begin{align}\label{eq:size M}
|M|\geq (1-2\heps^{1/2})n.
\end{align}
For every $i\in[r]$, define $X_i^H:=N_{ \Haux}(X_0^\sigma)\cap X_i$. Note that $|X_i^H|=(1\pm 3\heps^{1/2})|X_i|=(1\pm 4\heps^{1/2})n$.\COMMENT{$|X_i^H|\geq |X_0^\sigma|-\eps n\geq (1-2.5\heps^{1/2}) n \ge (1-3\heps^{1/2}) |X_i|$ by~\eqref{eq:size M}.}

\begin{substep}
Checking \ref{cond:4graphs i}
\end{substep}

In order to prove \ref{cond:4graphs i}, we first show that $\Agoodupd_i[X_i^H,V_i]$ is super-regular for every $i\in[r]$.
We will show that every vertex in $X_i^H\cup V_i$ has the appropriate degree, and that the common neighbourhood of most pairs of vertices in $V_i$ has the correct size, such that we can employ Theorem~\ref{thm: almost quasirandom} to guarantee the super-regularity of~$\Agoodupd_i[X_i^H,V_i]$.

For all $i\in[r]$ and for every vertex $x\in X_i^H$ with $\{x_0\}=N_{ \Haux}(x)\cap X_0^\sigma$, we have $\dg_{\Agoodupd_i[X_i^H,V_i]}(x)=|N_{\Agood_i}(x)\cap N_{\Ggood_{0i}}(\sigma(x_0))|$.
Hence, \eqref{eq:codegree F_12} implies that $$\dg_{\Agoodupd_i[X_i^H,V_i]}(x)=(d_i^Gd_i\pm\heps)|V_i|.$$
For $v\in V_i$, let $U_v:=N_{\Agood_i}(v)\subseteq X_i$. Observe that 
\begin{align}\label{eq:degree candidacy}
\dg_{\Agoodupd_i[X_i^H,V_i]}(v)=|\sigma(N_{ \Haux}(U_v)\cap X_0^\sigma)\cap N_{\Ggood_{0i}}(v)|, 
\end{align}
and $|N_{ \Haux}(U_v)\cap X_0|=|N_{\Agood_i}(v)|\pm\eps n=(d_i\pm2\heps)n$, and $|N_{\Ggood_{0i}}(v)|=(d_i^G\pm 2\heps)n$.
Adding for every $i\in[r]$ and every vertex $v\in V_i$, the weight function $\omega_{S,T}$ as defined in~\eqref{eq:def w_ST} for $S:=N_{ \Haux}(U_v)\cap X_0$ and $T:=N_{\Ggood_{0i}}(v)$ to $\cW$, we obtain that
\begin{align}\label{eq:superreg}
\dg_{\Agoodupd_i[X_i^H,V_i]}(v)
\stackrel{\eqref{eq:w_ST},\eqref{eq:degree candidacy}}{=}
(1\pm2\heps^{1/2}){|N_{ \Haux}(U_v)\cap X_0| |N_{\Ggood_{0i}}(v)|}n^{-1}
=(d_i^Gd_i\pm {\heps}^{1/3})|X_i^H|.
\end{align}
Note that these are at most $2rn$ weight functions $\omega_{S,T}$ that we added to $\cW$.

We will use Theorem~\ref{thm: almost quasirandom} to show that $\Agoodupd_i[X_i^H,V_i]$ is super-regular.
We call a pair of vertices $u,v\in V_i$ \emph{good} if $|N_{\Agood_i}(u,v)|=(d_i\pm\heps)^2|X_i|$, and
$|N_{\Ggood_{0i}}(u,v)|=(d_i^G\pm\heps)^2n$.
By the $\heps$-regularity of $\Agood_i$ and $\Ggood_{0i}$, using Fact~\ref{fact:regularity}, there are at most $2\heps |V_i|^2$ pairs $u,v\in V_i$ which are not good.
For every $i\in[r]$ and all good pairs $u,v\in V_i$, let $S_{u,v}:=N_{ \Haux}(N_{\Agood_i}(u,v))\cap X_0$ and $T_{u,v}:=N_{\Ggood_{0i}}(u,v)$. 
We add the weight function $\omega_{S_{u,v},T_{u,v}}$ as defined in~\eqref{eq:def w_ST} to $\cW$.
Observe that $|S_{u,v}|=|N_{\Agood_i}(u,v)|\pm\eps n=(d_i\pm2\heps)^2|X_i|$ and $|T_{u,v}|=(d_i^G\pm\heps)^2n$.
Note that these are at most $rn^2$ functions $\omega_{S_{u,v},T_{u,v}}$ that we add to $\cW$ in this way.
By~\eqref{eq:w_ST}, we obtain for all good pairs $u,v\in V_i$ that
\begin{align*}
\nonumber|N_{\Agoodupd_i[X_i^H,V_i]}(u,v)|
&=|\sigma(S_{u,v}\cap X_0^\sigma)\cap T_{u,v}|
=(1\pm 2\heps^{1/2})|S_{u,v}| |T_{u,v}|n^{-1}
\le (d_i^Gd_i+\heps^{1/3})^2|X_i^H|.
\end{align*}
Together with \eqref{eq:superreg}, we can apply Theorem~\ref{thm: almost quasirandom} and obtain that 
\begin{align}\label{eq:hat A superregular}
\text{$\Agoodupd_i[X_i^H,V_i]$ is $\big(\heps^{1/18},d_i^Gd_i\big)$-super-regular for every $i\in[r]$.}
\end{align}

In order to complete the proof of~\ref{cond:4graphs i}, for every $i\in[r]$, since $|X_i\sm X_i^H|\leq 3\heps^{1/2}|X_i|$, we can easily find a spanning subgraph $A_i^{new}$ of $\Agoodupd_i$ that is $(\eps',d_i^Gd_i)$-super-regular by deleting from every vertex $x\in X_i\sm X_i^H$ a suitable number of edges.
This establishes~\ref{cond:4graphs i}.

\begin{substep}
Checking \ref{cond:4graphs ii}
\end{substep}

Next, we show that for every $i\in[r]$, the edge set colouring $c^\sigma$ restricted to $\Agoodupd_i$ is $(1+\eps')d_i^Gd_i|X_i|$-bounded, which implies~\ref{cond:4graphs ii} because $A_i^{new}\subseteq A_i^\ast$.
Recall that we defined $c^\sigma$ (in~Definition~\ref{def:updated colouring}) such that for $xv\in E(\Agoodupd_i)$, we have $c^\sigma(xv)=c(xv)\cup c(\sigma(x_0)v)$ if $x$ has an $H$-neighbour $x_0\in X_0^\sigma$, and otherwise $c^\sigma(xv)=c(xv)$.
Since $c$ is colour-split, we may assume that $c_{\Agood_i}\colon E(\Agood_i)\to 2^{C_{\Agood_i}}$ is the edge set colouring $c$ restricted to $\Agood_i$ and $c_{\Ggood_{0i}}\colon E(\Ggood_{0i})\to C_{\Ggood_{0i}}$ is the edge-colouring~$c$ restricted to $\Ggood_{0i}$ such that $C_{\Agood_i}\cap C_{\Ggood_{0i}}=\emptyset$ for all $i\in[r]$. Fix $i\in[r]$. We have to show that for all $\alpha\in C_{A_i}\cup C_{\Ggood_{0i}}$, there are at most $(1+\eps')d_i^Gd_i|X_i|$ edges of~$\Agoodupd_i$ on which~$\alpha$ appears.

First, consider $\alpha\in C_{\Agood_i}$. Let $E_\alpha\subseteq E(\Agood_i)$ be the edges of $\Agood_i$ on which $\alpha$ appears. 
By assumption, $|E_\alpha|\leq (1+\eps)d_i|X_i|$. We need to show that $|E_\alpha\cap E(\Agoodupd_i)|\le (1+\eps')d_i^Gd_i|X_i|$.
To this end, we define a weight function $\omega_{\alpha}\colon E(\Agood_0)\to[\Lambda]_0$ by setting
$$\omega_\alpha(xv):=\big|\{v_i\in N_{\Ggood_{0i}}(v)\colon x_iv_i\in E_\alpha, xx_i \in E( \Haux[X_0,X_i]) \}\big|$$
for every $xv\in E(\Agood_0)$, and we add $\omega_\alpha$ to $\cW$. 
Note that 
\begin{align*}
|E_\alpha\cap E(\Agoodupd_i)| &\le \sum_{x_i\in X_i^H}\left|\set{v_i\in N_{\Ggood_{0i}}(\sigma(x))}{x_iv_i\in E_\alpha, xx_i\in E( \Haux[X_0,X_i])}\right|  +\Lambda|X_i\sm X_i^H| \\
   &\le \omega_{\alpha}(M)+ 3\heps^{1/2}\Lambda |X_i|.
\end{align*}
We now obtain an upper bound for $\omega_{\alpha}(M)$ using~\eqref{eq:w(M) upper bound}. 
For every edge $x_iv_i\in E_\alpha$ with $xx_i\in E( \Haux[X_0,X_i])$, condition~\eqref{eq:codegree F_13} states that
$$|N_{\Agood_0}(x)\cap N_{\Ggood_{0i}}(v_i)|=(d_i^Gd_0\pm\heps)n.$$
Hence, every such edge contributes weight $(d_i^Gd_0\pm\heps)n$ to $\omega_\alpha(E(\Agood_0))$. 
We obtain
\begin{align*}
&\omega_{\alpha}(E(\Agood_0))\leq (1+\eps)d_i|X_i|\cdot (d_i^Gd_0+\heps)n \le (d_0d_id_i^G+ 2\heps)|X_i|n.%\label{eq:weight A-colours}
\end{align*}
Now \eqref{eq:w(M) upper bound} implies that 
$\omega_{\alpha}(M)\leq (1+2\heps^{1/2})d_i^Gd_i|X_i|$ and hence $|E_\alpha\cap E(\Agoodupd_i)|\le (1+\eps')d_i^Gd_i|X_i|$.

\bigskip
Now, consider $\alpha\in C_{\Ggood_{0i}}$. Let $E_\alpha\subseteq E(\Ggood_{0i})$ be the set of edges of $\Ggood_{0i}$ on which $\alpha$ appears.
We define a weight function $\omega_{\alpha}\colon E(\Agood_0)\to[\Lambda]_0$ by setting
$$\omega_\alpha(xv):=\big|\{v_i\in N_{\Ggood_{0i}}(v)\colon vv_i\in E_\alpha, xx_i\in E(H[X_0,X_i]),x_iv_i\in E(\Agood_i) \}\big|$$
for every $xv\in E(\Agood_0)$, and we add $\omega_\alpha$ to $\cW$. Note that the number of edges of $\Agoodupd_i$ on which $\alpha$ appears is at most~$\omega_{\alpha}(M)$. 

In order to bound $\omega_{\alpha}(M)$, we again use~\eqref{eq:w(M) upper bound} and seek an upper bound for $\omega_{\alpha}(E(\Agood_0))$.
Since $c$ is $(1+\eps)e_G(V_0,V_i)/e_H(X_0,X_i)$-bounded on $G[V_0,V_i]$ by assumption, we have
$|E_\alpha|\leq (1+{\eps}^{1/2})d_i^G|X_i|n/e_H(X_0,X_i).$

For every edge $vv_i\in E_\alpha$ with $v_i\in V_i\sm V_i^{bad}$, condition~\eqref{eq:codegree F_23} implies that
\begin{align*}
e_H(N_{\Agood_0}(v), N_{\Agood_i}(v_i))=(d_0d_i\pm\heps)e_H(X_0,X_i).
\end{align*} 
Hence, every edge $vv_i\in E_\alpha$ with $v_i\in V_i\sm V_i^{bad}$ contributes weight $(d_0d_i\pm 
\heps)e_H(X_0,X_i)$ to $\omega_{\alpha}(E(\Agood_0))$. 
Since $\Delta(E_\alpha)\le \Lambda$ and $|V_i^{bad}|\leq 3r\eps n$, there are at most $3r\Lambda\eps n$ edges $vv_i\in E_\alpha$ with $v_i\in V_i^{bad}$, each of which contributes weight at most~$n$.
We conclude that
\begin{align*}
\omega_{\alpha}(E(\Agood_0))\leq \frac{(1+\eps^{1/2})d_i^G|X_i|n}{e_H(X_0,X_i)}\cdot (d_0d_i+ \heps)e_H(X_0,X_i)+3r\Lambda\eps n^2 
\leq (d_0d_id_i^G+ 2\heps)|X_i|n.
\end{align*}
Now \eqref{eq:w(M) upper bound} implies that 
$\omega_{\alpha}(M)\leq (1+\eps')d_i^Gd_i|X_i|$, completing the proof of~\ref{cond:4graphs ii}.

\begin{substep}
Checking \ref{cond:4graphs iii}
\end{substep}

Finally, we show that for all $i\in[r]$, $\alpha\in C_{\Ggood_{0i}}$ and $\beta\in C_{A_i'}$, the pair $\{\alpha,\beta \}$ appears on at most $n^\eps$ edges of $\Agoodupd_i$. This implies~\ref{cond:4graphs iii}, as the codegree of a pair in $C_{A_i'}$ is at most $K$ by assumption, and the codegree of a pair in $C_{\Ggood_{0i}}$ is~$0$. Fix $i\in[r]$, $\alpha\in C_{\Ggood_{0i}}$ and $\beta\in C_{A_i'}$.
Let 
\begin{align*}
E_{\alpha,\beta}:=\{v_0v_ix_i\colon v_0v_i\in E(\Ggood_{0i}), x_iv_i\in E(\Agood_i), c(v_0v_i)=\Set{\alpha}, \beta\in c(x_iv_i) \}
\end{align*}
and define the weight function $\omega_{\alpha,\beta}\colon E(\Agood_0)\to[\Lambda]_0$ by setting
\begin{align*}
\omega_{\alpha,\beta}(xv):=\big|\{vv_ix_i\in E_{\alpha,\beta} \colon xx_i\in E( \Haux[X_0,X_i])  \}\big|.
\end{align*}
Note that the number of edges of $\Agoodupd_i$ on which $\{\alpha,\beta\}$ appears is at most~$\omega_{\alpha,\beta}(M)$. In order to bound $\omega_{\alpha,\beta}(M)$, note that every triple $vv_ix_i\in E_{\alpha,\beta}$ contributes weight at most~$1$ to $\omega_{\alpha,\beta}(E(\Agood_0))$.
By assumption, $c$ is locally $\Lambda$-bounded and (globally) $(1+\eps)d_i|X_i|$-bounded on $A_i$, which implies that
$\omega_{\alpha,\beta}(E(\Agood_0))\le |E_{\alpha,\beta}|\leq (1+\eps)d_i\Lambda |X_i|\leq 2\Lambda n$.
Now, \eqref{eq:w(M) upper bound} implies that $\omega_{\alpha,\beta}(M) \leq n^\eps.$
Hence, for all $i\in[r]$, $\alpha\in C_{\Ggood_{0i}}$ and $\beta\in C_{A_i'}$, we add the corresponding weight function $\omega_{\alpha,\beta}$ to~$\cW$, which implies~\ref{cond:4graphs iii}. 
This completes the proof.
\end{proof}

%%%%%%%%%%%%%%%%%%%%%%%%%%%%%%%%%%%%%%%%%%%%%%%%%%%%%%%%%%%%%%%%%%%%%%%%%%%%%%
\section{Proof of Lemma~\ref{lem:main}}\label{sec:blow-up}
In this section, we prove our rainbow blow-up lemma (Lemma~\ref{lem:main}).
First, we will deduce Lemma~\ref{lem:main} from a similar statement (Lemma~\ref{lem:embedding}), where we impose stronger conditions on $G$ and $H$. This reduction utilises the results of Section~\ref{sec:colour splitting}.
We will conclude with the proof of Lemma~\ref{lem:embedding}.

\begin{lemma}\label{lem:embedding}
Let $1/n\ll\eps\ll\gamma,d,1/r,1/\Lambda$.
Let $(H,G,(X_i)_{i\in[r]},(V_i)_{i\in[r]})$ be an $(\eps,d)$-super-regular blow-up instance. Assume further that
\begin{enumerate}[label={\rm (\roman*)}]
%\item $\Delta(H)\leq\Delta$, and
\item $|V_i|=(1\pm\eps)n$ for all $i\in[r]$;
\item for all $ij\in\binom{[r]}{2}$, the graph $H[X_i,X_j]$ is a matching of size at least $\gamma^2n$;
\item\label{8.1 colour-split} $c\colon E(G)\to C$ is a colour-split edge-colouring of $G$ such that $c$ is locally $\Lambda$-bounded and $c$ restricted to $G[V_i,V_j]$ is {$(1-\gamma)e_G(V_i,V_j)/e_H(X_i,X_j)$}-bounded for all $ij\in\binom{[r]}{2}$.
\end{enumerate}
Then there exists a rainbow embedding $\phi$ of $H$ into $G$ such that $\phi(x)\in V_i$ for all $i\in[r]$ and $x\in X_i$.
\end{lemma}

\lateproof{Lemma~\ref{lem:main}}
We split the proof into three steps. In Step~\ref{step:1}, we apply Lemma~\ref{lem:colour splitting} in order to obtain a spanning subgraph $G_1\subseteq G$ such that the restricted edge-colouring is colour-split.
In Step~\ref{step:2}, we apply Lemma~\ref{lem:Hajnal Szemeredi partitioning} in order to refine the partitions of $G_1$ and $H$ in such a way that the vertex classes of $H$ are $2$-independent. Then, in Step~\ref{step:3}, we can apply Lemma~\ref{lem:embedding} to complete the proof.

In view of the statement, we may assume that $1/n\ll \eps \ll \gamma\ll d,1/r,1/\Delta,1/\Lambda$.
Choose new constants $\eps_1,\eps_2,\gamma',d_1,d_2$ with $\eps\ll\eps_1\ll\eps_2\ll d_2\ll\gamma'\ll d_1\ll\gamma$.

\begin{step}\label{step:1}
Colour-splitting
\end{step}

First, let $H_1$ be a supergraph of $H$ on $V(H)$ such that $e_{H_1-H}(X_i,X_j)\le \gamma^2n \le e_{H_1}(X_i,X_j)$ for all $ij\in\binom{[r]}{2}$ and $\Delta(H_1)\leq \Delta':=\Delta+r$.\COMMENT{Let $H'$ be a graph consisting of a matching of size $\gamma^2n$ between each pair of clusters. Let $H_1$ be the (simple) union of $H$ and $H'$. Then clearly the edge condition is satisfied and $\Delta(H_1)\le \Delta(H)+\Delta(H')\le \Delta + r$}
We claim that for all $\alpha\in C$, we have
$$\sum_{ij\in \binom{[r]}{2}}e^\alpha_G(V_i,V_j)e_{H_1}(X_i,X_j)\leq \left(1-\frac{\gamma}{2}\right)dn^2.$$
Indeed, since $c$ is locally $\Lambda$-bounded, we obtain that $e^\alpha_G(V_i,V_j)e_{H_1-H}(X_i,X_j)
\leq 2\Lambda n \cdot \gamma^2 n 
\leq r^{-2} \cdot \gamma dn^2/2$  for each $ij\in \binom{[r]}{2}$.
Hence, we can apply Lemma~\ref{lem:colour splitting} to $(H_1,G,(X_i)_{i\in[r]},(V_i)_{i\in[r]})$ (with $\gamma/2,\Delta'$ playing the roles of $\gamma,\Delta$), and obtain a spanning subgraph $G_1$ of $G$ such that $(H_1,G_1,(X_i)_{i\in[r]},(V_i)_{i\in[r]})$ is an $(\eps_1,d_1)$-super-regular blow-up instance, and the colouring $c_1:=c|_{E(G_1)}$ is colour-split and
$$\left(1-\frac{\gamma}{4}\right)\frac{e_{G_1}(V_i,V_j)}{e_{H_1}(X_i,X_j)}\text{-bounded}$$
for each bipartite subgraph $G_1[V_i,V_j]$.
Clearly, a rainbow embedding of $H_1$ into $G_1$ also yields a rainbow embedding of $H$ into $G$.

\begin{step}\label{step:2}
Refining the vertex partitions
\end{step}

We can now apply Lemma~\ref{lem:Hajnal Szemeredi partitioning} to the $(\eps_1,d_1)$-super-regular blow-up instance $(H_1,G_1,(X_i)_{i\in[r]},(V_i)_{i\in[r]})$ with edge-colouring $c_1$ and $\gamma',\Delta'$ playing the roles of $\gamma,\Delta$.
Hence, we obtain an $(\eps_2,d_2)$-super-regular blow-up instance $(H_2,G_2,(X_{i,j})_{i\in [r],j\in[\Delta'^2]},(V_{i,j})_{i\in [r],j\in[\Delta'^2]})$ such that for $n':=n/\Delta'^2$ we have that

\begin{enumerate}[label=(\alph*)]
\item $(X_{i,j})_{j\in[\Delta'^2]}$ is partition of $X_i$ and $(V_{i,j})_{j\in[\Delta'^2]}$ is partition of $V_i$ for every $i\in[r]$, and $|X_{i,j}|=|V_{i,j}|=(1\pm\eps_2)n'$ for all $i\in[r], j\in[\Delta'^2]$;
\item $H_2$ is a supergraph of $H_1$ on $V(H)$ such that $H_2[X_{i_1,j_1},X_{i_2,j_2}]$ is a matching of size at least $\gamma'^4 n'$ for all $i_1,i_2\in[r], j_1,j_2\in[\Delta'^2], (i_1,j_1)\neq(i_2,j_2)$;
\item $G_2$ is a graph on $V(G)$ such that $G_2[V_{i_1,j_1},V_{i_2,j_2}]\subseteq G_1[V_{i_1},V_{i_2}]$ for all distinct $i_1,i_2\in[r]$ and all $j_1,j_2\in[\Delta'^2]$;
\item $c_2$ is an edge-colouring of $G_2$ such that $c_2|_{E(G_1)\cap E(G_2)}=c_1|_{E(G_1)\cap E(G_2)}$, and $c_2$ is colour-split with respect to the partition $(V_{i,j})_{i\in [r],j\in[\Delta'^2]}$, and $c_2$ is locally $\Lambda$-bounded, and $c_2$ restricted to $G_2[V_{i_1,j_1},V_{i_2,j_2}]$ is $$\left(1-\frac{\gamma'}{2}\right)\frac{e_{G_2}(V_{i_1,j_1},V_{i_2,j_2})}{e_{H_2}(X_{i_1,j_1},X_{i_2,j_2})}\text{-bounded}$$ for all $i_1,i_2\in[r], j_1,j_2\in[\Delta'^2], (i_1,j_1)\neq(i_2,j_2)$.
\end{enumerate}
Again, a $c_2$-rainbow embedding of $H_2$ into $G_2$ also yields a $c_1$-rainbow embedding of $H_1$ into~$G_1$.

\begin{step}\label{step:3}
Applying Lemma~\ref{lem:embedding}
\end{step}

We can now complete the proof by applying Lemma~\ref{lem:embedding} as follows:
\begin{center}
\begin{tabular}{c|c|c|c|c|c|c|c|c|c}
 $n'$ & $\eps_2$ & $\gamma'^2$ & $d_2$ & $r\Delta'^2$ & $\Lambda$ & $H_2$ &  $G_2$ &  $(X_{i,j})_{i\in [r],j\in[\Delta'^2]}$ &  $(V_{i,j})_{i\in [r],j\in[\Delta'^2]}$ \\ \hline
 $n$ & $\eps$ & $\gamma$ & $d$ & $r$ & $\Lambda$ & $H$ &  $G$ & $(X_i)_{i\in[r]}$ & $(V_i)_{i\in[r]}$\vphantom{\Big(}
\end{tabular}
\end{center}
This yields a rainbow embedding of $H_2$ into $G_2$, and hence of $H$ in $G$.
\endproof

We now deduce Theorem~\ref{thm:quasirandom} from Lemma~\ref{lem:main} by partitioning $H$ using the Hajnal--Szemer\'edi theorem (Theorem~\ref{thm:HS}) and $G$ randomly.

\lateproof{Theorem~\ref{thm:quasirandom}} 
Let $r:=\Delta+1$. We may assume that $\eps$ is sufficiently small and $n$ is sufficiently large.
By applying Theorem~\ref{thm:HS} to $H$, we obtain a partition $(X_i)_{i\in[r]}$ of $V(H)$ into independent sets with $|X_i|\in\{\lfloor\frac{n}{r}\rfloor,\lceil\frac{n}{r}\rceil\}$.
We claim that there exists a partition $(V_i)_{i\in[r]}$ of~$V(G)$ such that 
\begin{enumerate}[label={\rm (\roman*)}]
  \item\label{quasi sr} $G[V_i,V_j]$ is $(2r\eps,d)$-super-regular for all $ij\in\binom{[r]}{2}$;
	\item \label{quasi colours} for all $\alpha\in C$ with $e^\alpha(G)\ge n^{3/4}$, we have $e^\alpha_G(V_i,V_j)=(1\pm\eps)2e^\alpha(G)/r^2$ for all $ij\in\binom{[r]}{2}$;
	\item\label{quasi sizes} $|V_i|=|X_i|$ for all $i\in[r]$.
\end{enumerate}
That such a partition exists can be seen using a probabilistic argument: For each $v\in V(G)$ independently, choose a label $i\in[r]$ uniformly at random and put $v$ into~$V_i$.
Using Chernoff's inequality (Theorem~\ref{thm:general chernoff simple}) for~\ref{quasi sr} and McDiarmid's inequality (Theorem~\ref{thm:McDiarmid}) for~\ref{quasi colours}, it is easy to check that~\ref{quasi sr} and~\ref{quasi colours} are satisfied with probability at least $1-\eul^{-n^{1/3}}$. Moreover, \ref{quasi sizes} holds with probability $\Omega(n^{-r/2})$. Hence, such a partition exists.

Therefore, we conclude that $(H,G,(X_i)_{i\in[r]},(V_i)_{i\in[r]})$ is a $(2r\eps,d)$-super-regular blow-up instance. Consider $\alpha\in C$. If $e^\alpha(G)\le n^{3/4}$, then condition~\ref{cond3 main} in Lemma~\ref{lem:main} clearly holds. If $e^\alpha(G)\ge n^{3/4}$, we use \ref{quasi colours} to see that
\begin{align*}
\sum_{ij\in \binom{[r]}{2}} e^\alpha_G(V_i,V_j)e_H(X_i,X_j)  &= (1\pm \eps)2e^\alpha(G)e(H)/r^2 \le (1+\eps)(1-\gamma)2e(G)/r^2 \le (1-\gamma/2)d(n/r)^2.
\end{align*}
Thus, we can apply Lemma~\ref{lem:main} and obtain a rainbow copy of $H$ in $G$.
\endproof

It remains to prove Lemma~\ref{lem:embedding}. The proof splits into four steps as follows.
In Step~\ref{step:colour splitting}, we split~$G$ into two spanning subgraphs $G_A$ and $G_B$ with disjoint colour sets.
In Step~\ref{step:candidacy graphs}, we define the necessary `candidacy graphs' that we track during the approximate embedding in Step~\ref{step:induction}. 
We then iteratively apply Lemma~\ref{lem:embedding lemma} in Step~\ref{step:induction} to find approximate rainbow embeddings of $X_i$ into $V_i$ using only the edges of $G_A$.
All those steps have to be performed carefully such that we can employ Lemma~\ref{lem:blow up matchings} in Step~\ref{step:absorbing} and use the reserved set of colours of $G_B$ to turn the approximate rainbow embedding into a complete one.

\lateproof{Lemma~\ref{lem:embedding}}
In view of the statement, we may assume that $\gamma\ll d,1/r,1/\Lambda$.
Choose new constants $\eps_0,\eps_1,\ldots,\eps_{r+1},\mu$ with $\eps\ll\eps_0\ll\eps_1\ll\cdots\ll\eps_{r+1}\ll\mu\ll\gamma$.\COMMENT{deleted $\eps'$ and replaced occurences with $\eps_0$}
For $i\in[r]$, let 
\begin{align*}
\bX_i  &:=\textstyle\bigcup_{j\in [i]}X_j,            &           \bV_i  &:=\textstyle\bigcup_{j\in[i]}V_j.              
\end{align*}

\begin{step}\label{step:colour splitting}
Colour splitting
\end{step}

In order to reserve an exclusive set of colours for the application of Lemma~\ref{lem:blow up matchings}, we randomly partition the edges of $G$ into two spanning subgraphs $G_A$ and $G_B$ as follows.
For each colour class of $G$ independently, we add its edges to $G_A$ with probability $1-\gamma$ and otherwise to~$G_B$. 
Let $d_A:=(1-\gamma)d$ and $d_B:=\gamma d$.
By Lemma~\ref{lem:regularity colour splitting},\COMMENT{where $Y_\alpha$ is a Bernoulli random variable with parameter $1-\gamma$ (or $\gamma$) for every colour $\alpha$} we conclude that with probability at least $1-1/n$,
 \begin{align}\label{eq:G_Z superregular}
 \begin{minipage}[c]{0.8\textwidth}\em
$G_Z[V_i,V_j]$ is $(\eps_0^2,d_Z)$-super-regular for all $ij\in\binom{[r]}{2}$, $Z\in\{A,B\}$.
 \end{minipage}\ignorespacesafterend
\end{align}
Hence, we may assume that $G$ is partitioned into $G_A$ and $G_B$ such that~\eqref{eq:G_Z superregular} holds.

\begin{step}\label{step:candidacy graphs}
Candidacy graphs
\end{step}

We want to show that there is a partial rainbow embedding of $H[\bX_r]$ into $G_A[\bV_r]$ that maps almost all vertices of $X_i$ into $V_i$ for every $i\in[r]$. Moreover, we need to ensure certain conditions for the remaining unembedded vertices in order to finally apply Lemma~\ref{lem:blow up matchings}.
We will achieve this by iteratively applying the Approximate Embedding Lemma (Lemma~\ref{lem:embedding lemma}) in Step~\ref{step:induction}. 
In order to formally state the induction hypothesis, we need some preliminary definitions.

For $t\in[r]_0$, we call $\phi_t\colon X_1^{\phi_t}\cup\ldots\cup X_t^{\phi_t}\to V_1^{\phi_t}\cup\ldots\cup V_t^{\phi_t}$ a \emph{$t$-partial embedding} if $X_i^{\phi_t}\subseteq X_i$, $V_i^{\phi_t}\subseteq V_i$, and $\phi_t(X_i^{\phi_t})=V_i^{\phi_t}$ for every $i\in[t]$, such that $\phi_t$ is an embedding of $H[X_1^{\phi_t}\cup\ldots\cup X_t^{\phi_t}]$ into $G_A[V_1^{\phi_t}\cup\ldots\cup V_t^{\phi_t}]$.
For brevity, define
\begin{align*}
\bX_t^{\phi_t} &:=\textstyle\bigcup_{i\in [t]}X_i^{\phi_t},   &           
\bV_t^{\phi_t} &:=\textstyle\bigcup_{i\in [t]}V_i^{\phi_t}.
\end{align*}

Given a $t$-partial embedding $\phi_t$, we define two kinds of bipartite auxiliary graphs: For each $i\in[r]\sm[t]$, we define a graph $A_i(\phi_t)$ with bipartition $(X_i,V_i)$ that tracks the still available images of a vertex $x\in X_i$ in~$G_A$, which will be used to extend the $t$-partial rainbow embedding~$\phi_t$ to a $(t+1)$-partial rainbow embedding $\phi_{t+1}$ via Lemma~\ref{lem:embedding lemma} in Step~\ref{step:induction}. Moreover, for each $i\in[r]$, we define a bipartite graph $B_i(\phi_t)$ that tracks the potential images of a vertex $x\in X_i$ in~$G_B$, which will be used for the completion via Lemma~\ref{lem:blow up matchings} in Step~\ref{step:absorbing}. Here, we keep tracking potential images of vertices even if they have been embedded, since in Step~\ref{step:absorbing}, we will actually `unembed' a few vertices.

When extending $\phi_t$ to $\phi_{t+1}$, we intend to update the graphs $A_i(\phi_t)$ and $B_i(\phi_t)$ simultaneously using Lemma~\ref{lem:embedding lemma}.
In order to facilitate this, we define $B_i(\phi_t)$ on a copy $(X_i^B,V_i^B)$ of the bipartition $(X_i,V_i)$.
For every $i\in[r]$, let $X_i^B$ and $V_i^B$ be disjoint copies of $X_i$ and $V_i$, respectively.
Let $\pi$ be the bijection that maps a vertex in $\bigcup_{i\in[r]}(X_i\cup V_i)$ to its copy in $\bigcup_{i\in[r]}(X_i^B\cup V_i^B)$. 
Let~$G^+$ and~$H^+$ be supergraphs of $G_A$ and~$H$ with vertex partitions $(V_1,\ldots,V_r,V_1^B,\ldots,V_r^B)$ and $(X_1,\ldots,X_r,X_1^B,\ldots,X_r^B)$, respectively, and edge sets
\begin{align*}
E(G^+) &:= E(G_A) \cup \left\{u\pi(v),v\pi(u)\colon uv\in E(G_B) \right\}\cup E_G^\ast, \\ 
E(H^+) &:= E(H) \cup \left\{x\pi(y),y\pi(x)\colon xy\in E(H) \right\}\cup E_H^\ast,
\end{align*}
where we added for convenience a suitable set $E_G^\ast\subseteq \bigcup_{i\in [r]}\{uv\colon u\in V_i,v\in V_i^B \}$ such that $G^+[V_i,V_i^B]$ is $(\eps_0,d_B)$-super-regular for all $i\in[r]$, and the set $E_H^\ast:=\{x\pi(x)\colon x\in V(H)\}$ so that $H^+[X_i,X_i^B]$ is a perfect matching for all $i\in[r]$.\COMMENT{In order to apply Lemma~\ref{lem:embedding lemma} to $G^+-\cV_t$ in the $(t+1)$-th step of the induction, we also need that $G^+[V_{t+1},V_{t+1}^B]$ is superregular to obtain an embedding instance, and $e_{H^+}(X_{t+1},X_{t+1}^B)\geq \eps_0 n$. Only due to technical reasons.}
Note that $G^+[V_i,V_j]=G_A[V_i,V_j]$, whereas $G^+[V_i,V_j^B]$ and $G^+[V_i^B,V_j]$ are isomorphic to $G_B[V_i,V_j]$ for all $ij\in\binom{[r]}{2}$.

We now define $A_i(\phi_t)$ and $B_i(\phi_t)$. Let $X_i^A:=X_i$ and $V_i^A:=V_i$ for every $i\in[r]$.
For $Z\in\{A,B\}$ and $i\in[r]$, we say that $v_i\in V_i^Z$ is a \emph{candidate for
$x_i\in X_i^Z$ (given~$\phi_t$)} if 
\begin{align}\label{eq:def candidates}
\phi_t(N_{H^+}(x_i)\cap \bX_t^{\phi_t})
\subseteq N_{G^+}(v_i),
\end{align}
\COMMENT{Note that if $N_{H^+}(x_i)\cap \bX_t^{\phi_t}=\emptyset$, then $v_i$ is a trivial candidate for $x_i$.} and we define $Z_i(\phi_t)$ as the bipartite graph with partition $(X_i^Z,V_i^Z)$ and edge set\COMMENT{Note that if $Z=A$, we could restrict $i$ to $i\in[r]\sm[t]$ since $A_i^t$ for $i\in[r]\sm[t]$ incorporate the `embedding' candidacy graphs for future rounds given $\phi_t$; whereas $B_i^t$ incorporate the `absorber' candidacy graphs such that also for $i\leq t$, the candidacy graph $B_i^t$ has to be updated when we extend $\phi_t$ to $\phi_{t+1}$.}
$$E(Z_i(\phi_t)):=\big\{x_iv_i\colon x_i\in X_i^Z, v_i\in V_i^Z, \text{ and $v_i$ is a candidate for $x_i$ given $\phi_t$} 
\big\}.$$
We call any spanning subgraph of $Z_i(\phi_t)$ a \emph{candidacy graph}.

Next, we define edge set colourings for these candidacy graphs.
For $i\in[r]\sm[t]$, we assign to every edge $e=x_iv_i\in E(A_i(\phi_t))$ a colour set $c_t(e)$ of size at most~$t$, which represents the colours that would be used if we were to embed $x_i$ at $v_i$ in the next step.
More precisely, for every $i\in[r]\sm[t]$ and every edge $x_iv_i\in E(A_i(\phi_t))$, we set
\begin{align}\label{eq:c_t}
c_t(x_iv_i):= c\big(E\big(G_A\big[\phi_t(N_{H}(x_i)\cap\bX_t^{\phi_t}),\{v_i\}\big]\big)\big).
\end{align}
Tracking this set will help us to ensure that the embedding is rainbow when we extend $\phi_t$ to~$\phi_{t+1}$. 
Since $|N_{H}(x_i)\cap\bX_t^{\phi_t}|\leq t$ and $|c(e)|=1$ for all $e\in E(G_A)$, we have $|c_t(x_iv_i)|\leq t$.

For the candidacy graphs $B_i(\phi_t)$, we merely need to know that they maintain super-regularity during the inductive approximate embedding (see \ind{t} below). Hence, for convenience, we set $c_t(e):=\emptyset$ for every $e\in E(B_i(\phi_t))$.

We also assign artificial dummy colours to the edges of $E(G^+)\sm E(G_A)$ as follows.
Let $c^{art}\colon\binom{V(G^+)}{2}\to C^{art}$ be a rainbow edge-colouring of all possible edges in $V(G^+)$ such that $C^{art}\cap C=\emptyset$. Define $c^+$ on $E(G^+)$ by setting $c^+(e):=c(e)$ if $e\in E(G_A)$ and $c^+(e):=c^{art}(e)$ otherwise.

\begin{step}\label{step:induction}
Induction
\end{step}

We inductively prove the following statement \ind{t} for all $t\in[r]_0$.
\begin{itemize}
\item[\ind{t}.] There exists a $t$-partial rainbow embedding $\phi_t\colon \bX_t^{\phi_t}\to \bV_t^{\phi_t}$ with $|X_s^{\phi_t}|=|V_s^{\phi_t}|\geq (1-\eps_t)|X_s|$ for all $s\in [t]$, and for all $Z\in\{A,B\}$, there exists a candidacy graph\COMMENT{so in particular spanning} $Z_i^t\subseteq Z_i(\phi_t)$ such that 
\begin{enumerate}[label=(\alph*)]
\item\label{candidacy superregular}$A_i^t$ is $(\eps_t,d_A^t)$-super-regular for all $i\in[r]\sm[t]$;
\item\label{B-candidacy superregular} $B_i^t$ is $(\eps_t,d_B^{t})$-super-regular for all $i\in[r]$;
\item\label{c_t bounded} the colouring $c_t$ restricted to $A_i^t$ is $(1+\eps_t)d_A^t|X_i|$-bounded and has codegree at most $n^{1/3}$ for all $i\in[r]\sm[t]$.
\end{enumerate}
\end{itemize}
\medskip
The statement \ind{0} holds for $\phi_0$ being the empty function:
Clearly, for all $Z\in\{A,B\}$, $i\in[r]$, the candidacy graph $Z_i(\phi_0)$ is complete bipartite, and by~\eqref{eq:c_t}, we have $c_0(e)=\emptyset$ for all $e\in E(A_i(\phi_0))$, implying \ind{0}. 

Hence, we may assume the truth of \ind{t} for some $t\in[r-1]_0$ and let $\phi_t\colon \bX_t^{\phi_t}\to \bV_t^{\phi_t}$ and $A_i^t,B_i^t$ be as in \ind{t}.
We will now extend $\phi_t$ to a $(t+1)$-partial rainbow embedding  $\phi_{t+1}$ such that \ind{t+1} holds.
Note that any matching $\sigma\colon X_{t+1}^\sigma\to V_{t+1}^\sigma$ in $A^t_{t+1}$ with $X_{t+1}^\sigma \subseteq X_{t+1}$ and $V_{t+1}^\sigma\subseteq V_{t+1}$  induces an embedding $\phi_{t+1}\colon \bX_{t}^{\phi_t}\cup X_{t+1}^\sigma \to \bV_{t}^{\phi_t}\cup V_{t+1}^\sigma $ which extends $\phi_t$ to a $(t+1)$-partial embedding as follows:
\begin{align}
\phi_{t+1}(x):=\begin{cases}
\phi_t(x) & \mbox{if } x\in \bX_{t}^{\phi_t}, \\
\sigma(x) & \mbox{if } x\in X_{t+1}^\sigma.
\end{cases} \label{def new phi}
\end{align}
The following is a key observation: Since $c$ is colour-split and by definition of the candidacy graph $A^t_{t+1}$ and the colouring~$c_t$ on $E(A^t_{t+1})$, whenever $\sigma$ is a \emph{rainbow} matching in $A^t_{t+1}$, then $\phi_{t+1}$ is a $(t+1)$-partial \emph{rainbow} embedding.

Now, we aim to apply Lemma~\ref{lem:embedding lemma} in order to obtain an almost perfect rainbow matching~$\sigma$ in~$A_{t+1}^t$. 
Let $H^{t+1}:= H^+-\bX_{t}$ and let $G^{t+1}:=G^+-\bV_{t}$. We claim that 
\begin{align}\label{embedding instance}
 \begin{minipage}[c]{0.8\textwidth}\em
$(H^{t+1},G^{t+1},\cA,c^+\cup c_t)$ is an $(\eps_t,\bd,(d_A^t,\bd^t),t,\Lambda)$-embedding-instance,
 \end{minipage}\ignorespacesafterend
\end{align}
where $\cA:=(A^t_{t+1},\ldots,A^t_r,B_1^t,\ldots,B_r^t)$ and $\bd:=(d_A,\ldots,d_A,d_B,\ldots,d_B)$ ($d_A$ repeated $r-t-1$ times and $d_B$ repeated $r$ times).

First, note that the colouring $c^+\cup c_t$ is locally $\Lambda$-bounded and colour-split with respect to the vertex partition $$(X_{t+1},\ldots,X_r,X_1^B,\ldots,X_r^B,V_{t+1},\ldots,V_r,V_1^B,\ldots,V_r^B)$$ of $G^{t+1}\cup\bigcup_{i\in [r]\sm[t]}A_i^t\cup\bigcup_{i\in [r]}B_i^t$. Moreover, the colour sets of $G^{t+1}$-edges have size~$1$ and the colour sets of candidacy graph edges have size at most~$t$.

Further, the super-regularity of the $G^{t+1}$-pairs follows from~\eqref{eq:G_Z superregular} (and for the pair $G^{t+1}[V_{t+1},V_{t+1}^B]$ from the choice of~$E_G^\ast$). 
Moreover, combining~\eqref{eq:G_Z superregular} with assumption~\ref{8.1 colour-split}, we infer that for every $i\in[r-t-1]$, the edge-colouring 
\begin{align*}
\text{$c$ restricted to $G_A[V_{t+1},V_{t+1+i}]$ is $(1+\eps_t)\frac{e_{G_A}(V_{t+1},V_{t+1+i})}{e_{H}(X_{t+1},X_{t+1+i})}$-bounded}.
\end{align*}
\COMMENT{By assumption, $c$ restricted to $G_A[V_i,V_j]$ is $(1-\gamma)\frac{e_G(V_i,V_j)}{e_H(X_i,X_j)}$-bounded; $(1-\gamma)e_G(V_i,V_j)\leq (1+\eps_t)e_{G_A}(V_i,V_j)$ since $e_G(V_i,V_j)\le (d+\eps)|V_i||V_j|$ and $e_{G_A}(V_i,V_j)\ge ((1-\gamma)d-\eps_0^2)|V_i||V_j|$ and $\frac{(1-\gamma)d-\eps_0^2}{d+\eps} \ge \frac{1-\gamma}{1+\eps_t}$.}

Finally, the super-regularity of the candidacy graphs and the boundedness of their colourings follows from~\ind{t}.
We conclude that \eqref{embedding instance} holds.
Hence, we can apply Lemma~\ref{lem:embedding lemma} to this instance with the following parameters:

\medskip
{
\noindent
{
\begin{center}
\begin{tabular}{c|c|c|c|c|c|c|c|c}
 $|X_{t+1}|$ & $\eps_t$ & $\eps_{t+1}$  & $t$ & $r-t-1+r$ & $\Lambda$ & $n^{1/3}$ & $\bd$ & $(d_A^t,\bd^t)$  \\ \hline
 $n$ & $\eps$ & $\eps'$ & $t$ & $r$ & $\Lambda$ & $K$ & $(d_i^G)_{i\in[r]}$ & $(d_i)_{i\in[r]_0}$ \vphantom{\Big(}
\end{tabular}
\end{center}
}
}
\medskip

Let $\sigma\colon X_{t+1}^\sigma\to V_{t+1}^\sigma$ be the rainbow matching in $A_{t+1}^t$ obtained from Lemma~\ref{lem:embedding lemma} with $|X_{t+1}^\sigma|\geq (1-\eps_{t+1})|X_{t+1}|$.
The matching $\sigma$ extends $\phi_t$ to a $(t+1)$-partial rainbow embedding $\phi_{t+1}$ as defined in~\eqref{def new phi}.
By Definition~\ref{def:updated candidacy}, the updated candidacy graphs with respect to $\sigma$ obtained from Lemma~\ref{lem:embedding lemma} are also updated candidacy graphs with respect to $\phi_{t+1}$ as defined in Step~\ref{step:candidacy graphs}. (More precisely, we have $Z_i^{t,\sigma}\In Z_i(\phi_{t+1})$ for $Z\in \Set{A,B}$.)
Hence, by Lemma~\ref{lem:embedding lemma}, we obtain new candidacy graphs $A_i^{t+1}\subseteq A_i(\phi_{t+1})$ for $i\in[r]\sm[t+1]$ and $B_i^{t+1}\subseteq B_i(\phi_{t+1})$ for $i\in[r]$ that satisfy ~\ref{cond:4graphs i}--\ref{cond:4graphs iii}.
By \ref{cond:4graphs i}, we know that $A_i^{t+1}$ is $(\eps_{t+1},d_A^{t+1})$-super-regular for every $i\in[r]\sm[t+1]$, and $B_i^{t+1}$ is $(\eps_{t+1},d_B^{t+1})$-super-regular for every $i\in[r]$, which implies~\ind{t+1}\ref{candidacy superregular} and~\ind{t+1}\ref{B-candidacy superregular}.
Moreover, the new colouring $c_{t+1}$ as defined in~\eqref{eq:c_t} corresponds to the updated colouring as in Definition~\ref{def:updated colouring}, so we can assume that $c_{t+1}$ satisfies~\ref{cond:4graphs ii} and \ref{cond:4graphs iii}.
Thus, for every $i\in[r]\sm[t+1]$, the colouring $c_{t+1}$ restricted to $A_i^{t+1}$ is $(1+\eps_{t+1})d_A^{t+1}|X_{i}|$-bounded by~\ref{cond:4graphs ii}, and has codegree at most $n^{1/3}$ by~\ref{cond:4graphs iii}.
This implies \ind{t+1}\ref{c_t bounded}, and hence completes the inductive step.

\begin{step}\label{step:absorbing}
Completion
\end{step}

We may assume that $\phi_r\colon \bX_r^{\phi_r}\to \bV_r^{\phi_r}$ is an $r$-partial embedding fulfilling \ind{r} with $(\eps_r,d_B^r)$-super-regular candidacy graphs $B_i^r\subseteq B_i(\phi_r)$.
Recall that we defined the bipartite candidacy graphs $B_i^r$ on copies $(X_i^B,V_i^B)$ only to conveniently apply Lemma~\ref{lem:embedding lemma} in Step~\ref{step:induction}.
We now identify $B_i^r$ with a bipartite graph $B_i'$ on $(X_i,V_i)$ and edge set $E(B_i'):=\{x_iv_i\colon \pi(x_i)\pi(v_i)\in E(B_i^r) \}$. Hence, for each $i\in[r]$, $B_i'$ is $(\eps_r,d_B^r)$-super-regular and for every edge $x_iv_i\in E(B_i')$, we deduce from \eqref{eq:def candidates} that
\begin{align}\label{eq:def candidates final}
\phi_r(N_{H}(x_i)\cap \bX_r^{\phi_r})
\subseteq N_{G_B}(v_i).
\end{align}

We want to apply Lemma~\ref{lem:blow up matchings} in order to complete the embedding using the edges in $G_B$ and the candidacy graphs $(B_i')_{i\in[r]}$.
For every $i\in[r]$, let $\Vc_i:=V_i\sm V_i^{\phi_r}$ and $\Xc_i:=X_i\sm X_i^{\phi_r}$ be the sets of unused/unembedded vertices. Note that we have no control over these sets except knowing that they are very small. To be able to apply Lemma~\ref{lem:blow up matchings}, we now (randomly) add vertices that have already been embedded back to the unembedded vertices. That is, we will find sets $V_i'\supseteq \Vc_i$ and $X_i'\supseteq\Xc_i$ of size exactly $n_B:=\lceil\mu n\rceil$ (same size required for condition~\ref{lem:cond2} in Lemma~\ref{lem:blow up matchings}) such that $B_i'[X_i',V_i']$ is still super-regular.

For the application of Lemma~\ref{lem:blow up matchings}, we also have to ensure that not only the colouring $c$ restricted to $G_B[V_1'\cup\ldots \cup V_r']$ is sufficiently bounded (see property~\ref{G_B mu bounded} below), but also that the colouring $c$ restricted to $G_B$ between already embedded sets $V_i\sm V_i'$ and sets $V_j'$ used for the completion is sufficiently bounded (see property~\ref{G_B between mu bounded} below).
Therefore, for $i,j\in [r]$\COMMENT{before: $ij\in\binom{[r]}{2}$, but this not symmetric, need $r^2$ conditions}, let $G_B^{hit}[V_i\sm V_i',V_j']$ be the spanning subgraph of $G_B[V_i\sm V_i',V_j']$ containing those edges $v_iv_j\in E(G_B[V_i\sm V_i',V_j'])$ for which $\phi_r^{-1}(v_i)$ has an $H$-neighbour in~$X_j'$.
That is, $G_B^{hit}[V_i\sm V_i',V_j']$ contains all the edges between $V_i\sm V_i'$ and $V_j'$ that will potentially be used to extend the partial embedding when applying Lemma~\ref{lem:blow up matchings}.

We claim that sets $\Vc_i^+\subseteq V_i^{\phi_r}$ can be chosen such that, setting $\Xc_i^+:=\phi_r^{-1}(\Vc_i^+)$, $V_i':=\Vc_i\cup\Vc_i^+$, and $X_i':=\Xc_i\cup\Xc_i^+$, we have:
\begin{enumerate}[label={\rm (\alph*)}]
\item\label{G_B superregular} $G_B[V_i',V_j']$ is $(\eps_{r+1},d_B)$-super-regular for all $ij\in\binom{[r]}{2}$;
\item\label{B_i superregular} $B_i'[X_i',V_i']$ is $(\eps_{r+1},d_B^r)$-super-regular for every $i\in[r]$;\COMMENT{By Chernoff's inequality, we may conclude that every vertex has the correct degree. Note that for each vertex $v\in V_i^{\phi_r}$ each neighbour in $B_i$ (also $\phi^{-1}(v)$) survives independently at random with probability $p_i$.
The regularity follows since $B_i$ is $(\eps_r,d_B^r)$-super-regular.}
\item\label{G_B mu bounded} the colouring $c$ restricted to $G_B[V_1'\cup\ldots\cup V_r']$ is $\mu^{3/2}n$-bounded;\COMMENT{
Note that for $v\in V_i^{\phi_r}$ we have that $$\prob{v\in \Vc_i^+}\leq 2\mu.$$
Let $\alpha$ be a colour that appears on $G_B[V_i,V_j]$. 
Note that $\alpha$ appears on at most $2\Lambda n$ edges of $G_B[V_i,V_j]$.
Let $X$ be the number of $\alpha$-coloured edges in $G_B[\Vc_i^+,\Vc_j^+]$. 
We have
$$\expn{X}\leq 8\mu^2\Lambda n.$$
Using McDiarmid's inequality we conclude that 
$$\prob{X-\expn{X}>\mu^2\Lambda n}\leq \exp(-n^{4/5}).$$
Hence, with probability at least $1-\exp(-n^{4/5})$, we have that $\alpha$ appears at most $8\mu^2\Lambda rn$ times in $G_B[\Vc_i^+,\Vc_j^+]$, and thus, since $|\Vc_i|\leq 2\eps_rn$, at most $\mu^{3/2}n$ times in $G_B[V_i',V_j']$.
}
\item\label{G_B between mu bounded} the colouring $c$ restricted to $G_B^{hit}[V_i\sm V_i',V_j']$ is $\mu^{3/2}n$-bounded for all $i,j\in[r]$;
\item\label{size V_i'} $|V_i'|=|X_i'|=n_B$ for every $i\in[r]$.
\end{enumerate}

This can be seen with a probabilistic argument. Independently for every $i\in[r]$ and $v\in V_i^{\phi_r}$, let $v$ belong to $\Vc_i^+$ with probability $p_i:=(n_B-|\Vc_i|)/|V_i^{\phi_r}|$. We now show that~\ref{G_B superregular}--\ref{size V_i'} hold simultaneously with positive probability.

Note that $p_i=\mu\pm\sqrt{\eps_r}$.
Recall that $G_B[V_i,V_j]$ is $(\eps_0,d_B)$-super-regular, $B_i'$ is $(\eps_r,d_B^r)$-super-regular, $|\Vc_i|=|\Xc_i|\leq 2\eps_r n$, and $c$ is locally $\Lambda$-bounded. Using Chernoff's bound, it is routine to show that~\ref{G_B superregular} and~\ref{B_i superregular} hold with probability at least $1-\eul^{-\sqrt{n}}$, say. Note here that the regularity follows easily from the regularity of the respective supergraphs.\COMMENT{For this we might need that the vertex sets are not too small, but this also holds with high probability by Chernoff.}

We show next that also~\ref{G_B between mu bounded} holds with high probability. 
Let $i,j\in [r]$ and let $\alpha$ be a colour.
Let~$X$ be the number of $\alpha$-coloured edges $v_iv_j$ in $G_B[V_i\sm V_i',V_j'])$ for which $v_j\in \Vc_j^+$ and $\phi_r^{-1}(v_i)$ has an $H$-neighbour in~$\Xc_j^+$. 
Note that since $|\Vc_j|=|\Xc_j|\leq 2\eps_r n$ and $c$ is locally $\Lambda$-bounded, the number of $\alpha$-coloured edges $v_iv_j$ in $G_B[V_i\sm V_i',V_j'])$ for which $v_j\in \Vc_j$ or $\phi_r^{-1}(v_i)$ has an $H$-neighbour in~$\Xc_j$, is at most $4\Lambda\eps_r n$. Now, consider an edge $v_iv_j\in E(G_B[V_i^{\phi_r},V_j^{\phi_r}])$ with $\{x_j\}=N_H(\phi_r^{-1}(v_i))\cap X_j^{\phi_r}$. 
Crucially, observe that $x_j\neq \phi_r^{-1}(v_j)$ because $v_iv_j$ is an edge in $G_B$ and therefore not in~$G_A$.\COMMENT{which implies that $\phi_r$ would not be a valid embedding into $G_A$, a contradiction.} 
This implies that
\begin{align*}
\prob{v_j\in \Vc_j^+, x_j\in \Xc_j^+}= p_j^2 \leq 2\mu^2.
\end{align*}
Since $c$ is locally $\Lambda$-bounded, $\alpha$ appears on at most $2\Lambda n$ such edges $v_iv_j$ and hence
\begin{align*}
\expn{X}\leq 4\Lambda\mu^2 n.
\end{align*}
Since $c$ is locally $\Lambda$-bounded, an application of McDiarmid's inequality yields that, with probability at least~$1-\eul^{-n^{2/3}}$, we have $X\le 5\Lambda\mu^2 n$, which implies that the number of $\alpha$-coloured edges in $G_B^{hit}[V_i\sm V_i',V_j']$ is at most~$\mu^{3/2}n$.
Together with a union bound, we infer that~\ref{G_B between mu bounded} holds with probability at least $1-\eul^{-\sqrt{n}}$.

A similar (even simpler) argument using the local boundedness of $c$ and McDiarmid's inequality also works for~\ref{G_B mu bounded}.
Thus, a union bound implies that~\ref{G_B superregular}--\ref{G_B between mu bounded} hold simultaneously with probability at least $1-4\eul^{-\sqrt{n}}$.
Moreover, standard properties of the binomial distribution yield that $|\Vc_i^+|=n_B-|\Vc_i|$ (and thus, $|V_i'|=|X_i'|=n_B$) for all $i\in[r]$ with probability at least $\Omega(n^{-r/2})$. Hence, there exist such sets $X_i'$ and $V_i'$ for all $i\in[r]$ satisfying~\ref{G_B superregular}--\ref{size V_i'}.

Let
\begin{align*}
\cX_r' &:=\textstyle\bigcup_{i\in[r]}X_i',&
\cV_r' &:=\textstyle\bigcup_{i\in[r]}V_i',\\
X_0' &:= \cX_r\sm \cX_r', &
V_0' &:= \cV_r\sm \cV_r'.
\end{align*}
The restriction of $\phi_r$ to $X_0'$ clearly yields a rainbow embedding $\psi_0\colon X_0'\to V_0'$ of $H[X_0']$ into~$G_A[V_0']$.
Let $G':=G_B[\cV_r']\cup G_B^{hit}[V_0',\cV_r']$, and let $H'$ be the subgraph of $H$ with partition $(X_i')_{i\in[r]_0}$ that arises from $H$ by discarding all edges in $H[X_0']$.
(This is feasible since edges within $X_0'$ have already been embedded by~$\psi_0$.)
By~\ref{G_B superregular}~and~\ref{B_i superregular}, we have that $(H',G',(X_i')_{i\in[r]_0},(V_i')_{i\in[r]_0})$ is an $(\eps_{r+1},d_B)$-super-regular blow-up instance with exceptional sets $(X_0',V_0')$ and $(\eps_{r+1},d_B^r)$-super-regular candidacy graphs $(B_i'[X_i',V_i'])_{i\in[r]}$. Moreover, $c$ restricted to $G'$ is $\mu^{1/2}n_B$-bounded by~\ref{G_B mu bounded} and~\ref{G_B between mu bounded}, and all clusters have the same size $n_B$ by~\ref{size V_i'}. Further,
\begin{itemize}
\item from~\eqref{eq:def candidates final} and the definition of $G_B^{hit}$, it holds that for all $x\in X_0'$, $i\in[r]$ and $x_i\in N_{H'}(x)\cap X_i'$,
we have $N_{B_i'}(x_i)\subseteq N_{G'}(\psi_0(x))$;\COMMENT{Consider $v_i\in N_{B_i'}(x_i)$. By~\eqref{eq:def candidates final}, we have $\phi_r(x)\in N_{G_B}(v_i)$. Since $x\in X_0'$, we have $\phi_r(x)=\psi_0(x)$. Hence, $\psi_0(x)v_i\in E(G_B)$. Moreover, since $x$ has an $H$-neighbour in $X_i'$, we included $\psi_0(x)v_i$ in $G_B^{hit}$.}

\item for all $i\in [r]$, $x\in X_i'$, $v\in N_{B_i'}(x)$ and distinct $x_0,x_0'\in N_{H'}(x)\cap X_0'$, we have $c(\psi_0(x_0)v)\neq c(\psi_0(x_0')v)$ because $\psi_0(x_0)$ and $\psi_0(x_0')$ belong to different clusters of $(V_i)_{i\in[r]}$ and~$c$ is colour-split with respect to $(V_i)_{i\in[r]}$.
\end{itemize}
Hence, we can finally apply Lemma~\ref{lem:blow up matchings} to obtain a rainbow embedding $\psi$ of $H'$ into $G'$ which extends~$\psi_0$, such that $\psi(x)\in N_{B_i'}(x)$ for all $i\in[r]$ and $x\in X_i'$.
Since the colours of $c$ restricted to $G'\In G_B$ are distinct from the colours already used by~$\psi_0$, it holds that $\psi$ is a valid rainbow embedding of $H$ into~$G$.
This completes the proof.
\endproof

%
%
%\section*{Acknowledgement}
%The first author would like to thank the School of Mathematics of the University of Birmingham for the hospitality in Birmingham, where part of this research was done.

\bibliographystyle{../amsplain_v2.0customized}
\bibliography{../ReferencesLocal}

\end{document}